\documentclass[12pt,a4paper]{article}
\usepackage{graphicx}
\usepackage{amsmath}
\usepackage[cp1251]{inputenc}
\usepackage[T2A]{fontenc}

\usepackage[russian]{babel}
\usepackage{a4}
\usepackage{amsfonts}
\usepackage{amssymb}

\newtheorem{theorem}{Theorem} [section]

\newtheorem{lemma}[theorem]{Lemma}

\newenvironment{proof}[1][Proof]{\textbf{#1.} }{\ \rule{0.5em}{0.5em}}
\numberwithin{equation}{section}

\begin{document}

\bigskip

UDC № 517.956.45

\bigskip

\textbf{Classical solvability of multidimensional two-phase Stefan
problem for degenerate parabolic equations.}

\bigskip

\textbf{  S.P.Degtyarev}

\bigskip

\textbf{Institute of Applied Mathematics and Mechanics of Ukrainian
National Academy of Sciences}

E-mail: degtyar@i.ua

\bigskip

We prove locally in time the existence of a smooth solution for
multidimensional two-phase Stefan problem for degenerate parabolic
equations of the porous medium type. We establish also natural
H\"{o}lder class for the boundary conditions in the Cauchy-Dirichlet
problem for a degenerate parabolic equation.

\bigskip

Key words: free boundary, Stefan problem, classical solvability, porous medium
equation, degenerate parabolic equations.

\bigskip

\bigskip

\ \ \ \ \ \ \ \ \ \ \ \ \ \ \ \ \ \ \ \ \ \ \ \ \ \ \ \ \ \ \ \ \ \
\ \ \ \ \ \ \ \ \ \ \ \ \ \ \ \ \ \ To the memory of Professor
B.V.Bazaliy

\bigskip

\textbf{The final publication is available at Springer via}

 http://dx.doi.org/10.1007/s00030-014-0280-3

\bigskip

\section{Statement of the problem and the main result} \label{s1}

Classical solvability of the Stefan problem for uniformly parabolic
equations has been well studied - see for example papers \cite{1} -
\cite{5} and the references therein. At the same time, as has long
been known, the heat transfer model based on uniformly parabolic
equations, has some properties which can not be observed in the
reality, in particular, the infinite speed of propagation of
disturbances. We also know that more accurate model of the heat
transfer is the model which is based on degenerate parabolic
equations, such as equations of the form

\begin{equation}
u_{t}(x,t)=\nabla (|u|^{m-1}\nabla u(x,t))=f(x,t), \label{1.1}
\end{equation}
where $m>1$. As it is known, a short formulation of a classical
Stefan problem for the equation \eqref{1.1} is the equation

\begin{equation}
(\beta(u))_{t}=\nabla (|u|^{m-1}\nabla u(x,t))=0, \label{1.2}
\end{equation}
where $\beta(u)$ is a discontinuous function of the form

\[
\beta(u)=%
\genfrac{\{}{.}{0pt}{}{u,\quad\quad u\leq0,}{u+k,\quad u>0,}%
\]
where $ k> 0 $, is the latent heat of fusion (crystallization) and
the equation \eqref {1.2} is considered in the sense of
distributions. At that for a quasilinear equation \eqref{1.2} the
main unknown is, in fact, the interface $\{u = 0\}$, outside of
which the solution of \eqref{1.2} is smooth in view of the
well-known local theory of uniformly parabolic equations.

In its generalised formulation this problem was considered in a
number of papers, from which we mention, for example, \cite{V} -
\cite{DiBe}, and we do not pretend to be complete in this matter, as
the subject of our interest in this article is a smooth solution of
the problem.

As for the smooth solutions, in the case of one spatial variable
such problem for degenerate parabolic equations was considered in
\cite{6} - \cite{10}, where it was proved the existence of classical
solutions. See also \cite{11}.

The aim of this paper is the proof of the classical solvability of
the Stefan problem of the type \eqref{1.2} for a degenerate equation
in a multidimensional setting, that is, the proof of the existence
of smooth surface which is the interface $\{u = 0\}$, and the proof
of the smoothness of the solution up to the interface.

We now formulate a precise statement of the Stefan problem in a more
expanded than \eqref{1.2} form, as it is custom in the theory of
free boundary problems. Let $\Omega$ is a doubly connected domain in
$R^{N}$, whose boundary consists of two smooth connected surfaces
$\Gamma^{+}$ and $\Gamma^{-}$ without self-intersections, $\partial
\Omega = \Gamma^{+} \cup \Gamma^{-}$. Suppose, further, that $\Gamma
$ is a given smooth surface without self-intersections lying
strictly between $\Gamma^{+}$ and $\Gamma^{-}$ and separating the
domain $\Omega$ into two doubly connected subdomains $\Omega^{+}$
and $\Omega^{-}$, so that $\partial \Omega^{\pm} = \Gamma \cup
\Gamma^{\pm}$. For a fixed $T>0$ denote $\Omega_{T}=\Omega\times
(0,T)$, $\Omega^{\pm}_{T}=\Omega^{\pm}\times (0,T)$,
$\Gamma_{T}=\Gamma\times [0,T]$,
$\Gamma^{\pm}_{T}=\Gamma^{\pm}\times [0,T]$.

Denote by $S_{T}$ a smooth surface in the cylindrical domain
$\Omega_{T}$ in the space $(y, \tau) \in R^{N} \times [0, T]$, such
that at $\tau = 0$, it coincides with $\Gamma$, that is, $S_{T} \cap
\{\tau = 0\} = \Gamma $, $S_{T}$ does not intersect surfaces
$\Gamma^{\pm}_{T}$ and divides the $\Omega_{T}$ into two subdomains
$Q^{+}_{T}$ and $Q^{-}_{T}$, and the lateral boundaries of these
domains consist of $S_{T}$ and $\Gamma^{\pm}_{T}$ respectively.
Surface $S_{T}$ is unknown and has to be determined together with
the functions $u^{+}(y, \tau)$ and $u^{-}(y, \tau)$, which are
defined in $Q^{\pm}_{T}$ respectively. The triple $(S_{T}, u^{+},
u^{-})$ must satisfy the following conditions (we denote the
independent variables by the $(y, \tau)$ in view of the subsequent
change of variables):

\begin{equation}
\frac{\partial u^{\pm}}{\partial
\tau}-\nabla_{y}(a^{\pm}|u^{\pm}|^{m-1}\nabla_{y}u^{\pm})=0, \
(y,\tau)\in Q^{\pm}_{T}, \label{1.3}
\end{equation}

\begin{equation}
u^{+}(y,\tau)=u^{-}(y,\tau)=0, \  (y,\tau)\in S_{T},
 \label{1.4}
\end{equation}

\begin{equation}
a^{+}\sum_{i=1}^{N}\cos(\overrightarrow{N},y_{i})|u^{+}|^{m-1}u^{+}_{y_{i}}-
a^{-}\sum_{i=1}^{N}\cos(\overrightarrow{N},y_{i})|u^{-}|^{m-1}u^{-}_{y_{i}}=k
\cos (\overrightarrow{N},\tau), \ (y,\tau)\in S_{T},
 \label{1.5}
\end{equation}

\begin{equation}
u^{\pm}(y,\tau)=g^{\pm}(y,\tau), \  (y,\tau)\in \Gamma^{\pm}_{T},
 \label{1.6}
\end{equation}

\begin{equation}
u^{\pm}(y,0)=u^{\pm}_{0}(y).
 \label{1.7}
\end{equation}
Here $m>1$, $k>0$, $a^{\pm}>0$ are given constants, $g^{+}(y,\tau)$,
$g^{-}(y,\tau)$, $u^{+}_{0}(y)$, $u^{-}_{0}(y)$ are given functions,
at that
\begin{equation} \pm g^{\pm}(y,\tau)\geq \nu>0, \
(y,\tau)\in\Gamma^{\pm}_{T}; \ \ \pm u^{\pm}_{0}(y)>0, \ y\in
\Omega^{\pm}, \ u^{\pm}_{0}(y)=0, \ y\in \Gamma,
 \label{1.8}
\end{equation}
where $\nu$ is some positive constant: here and below we denoted by
the same symbols $\nu$, $\mu$, $C$ all absolute constants, or
constants that depend only on the given data of the problem. Note
that the conditions \eqref{1.4}, \eqref{1.5} are the three
independent conditions at unknown boundary $S_{T}$ arising from the
equation \eqref{1.2}.

To formulate the smoothness conditions which we impose on the data
of the problem, we introduce some weighted function spaces. First of
all, we use the standard Ho"lder space
$H^{l+\delta}(\overline{\Omega})\equiv
C^{l+\delta}(\overline{\Omega})$, $\delta \in (0,1)$, $l \in
\mathbb{N} \cup \{0\}$, with the norm $|u|^{(l + \delta)}_{\overline
{\Omega}}$, which are introduced in \cite{12}, and also spaces $H^{l
+ \delta, \frac{l + \delta}{2}}(\overline {\Omega_{T}})\equiv C^{l +
\delta, \frac{l + \delta}{2}}(\overline {\Omega_{T}})$ of functions
of $ (y, \tau) $ with the norm $ |u|^{(l + \delta)}_{\overline
{\Omega_ {T}}}$. In \cite{12} the surface of the corresponding
classes are also defined. We assume that the surfaces $\Gamma$,
$\Gamma^{\pm}$ belong to following classes

\begin{equation}
 \Gamma, \ \Gamma^{\pm} \ \in H^{4+\gamma}
 \label{1.9}
\end{equation}
with some $0< \gamma <1$. At the same time we suppose that the
functions $g^{\pm}$ in \eqref{1.6} are such that

\begin{equation}
h^{\pm}(y,\tau)\equiv |g^{\pm}(y,\tau)|^{m-1}g^{\pm}(y,\tau)\in
H^{4+\gamma,\frac{4+\gamma}{2}}(\Gamma^{\pm}_{T}).
 \label{1.10}
\end{equation}

Suppose, further, that $d^{+}(y)$ is a given function from $H^{2 +
\gamma} (\overline {\Omega^{+}})$, which models the distance from a
point $y \in \overline {\Omega^{+}}$ to the surface $ \Gamma $, that
is

\begin{equation}
\nu \leq d^{+}(y)/dist(y,\Gamma) \leq \nu^{-1}.
 \label{1.13}
\end{equation}
Note, that such function can be taken, for example, as the solution
of the following problem

\[
\Delta d^{+}(y)=-1, \ y\in \Omega^{+},
\]
\[
d^{+}(y)=0, \ y\in \Gamma, \ \ \ d_{+}(y)=1, \ y\in \Gamma^{+} .
\]
Let, further, $d^{-}(y)$ is an analogous function for the domain
 $\overline{\Omega^{-}}$.

Denote here and below

\begin{equation}
\alpha=\frac{m-1}{m}\in(0,1),\label{B0}%
\end{equation}
where $m>1$ is the exponent from the equation \eqref{1.3}.

We will use the spaces
$C^{2+\gamma,\frac{2+\gamma}{2}}_{s}(\overline{\Omega}^{\pm}_{T})$
from the paper \cite{12.1} (they are analogous to the corresponding
spaces from \cite{14}), where $0<\gamma<\alpha$ is some exponent,
and we require

\begin{equation}
0<\gamma<\alpha . \label{B01}%
\end{equation}

These spaces are defined in the following way. First we define the
spaces $C^{2+\gamma}_{s}(R^{N}_{+T})$ in the domain

\begin{equation}
R_{+T}^{N}=R_{+}^{N}\times\lbrack0,T],\quad R_{+}^{N}=\left\{
x=(x^{\prime
},x_{N}):x_{N}\geq0,x^{\prime}\in R^{N-1}\right\}  . \label{B1}%
\end{equation}
Define a distance between points $x,\overline{x}\in R^{N}_{+}$
according to the following formula

\begin{equation}
s(x,\overline{x})=\frac{\left|  x-\overline{x}\right|
}{x_{N}^{\alpha
/2}+\overline{x}_{N}^{\alpha/2}+|x^{\prime}-\overline{x}^{\prime}|^{\alpha/2}%
}. \label{B2}%
\end{equation}
Define further a H\"{o}lder constant of a function $u(x,t)$ with
respect to the variable $x$ according to the distance we have
introduced

\begin{equation}
H_{s,R_{+T}^{N}}^{\gamma}(u)\equiv\sup_{(x,t),(\overline{x},t)\in R_{+T}^{N}%
}\frac{|u(x,t)-u(\overline{x},t)|}{s(x,\overline{x})^{\gamma}}. \label{B3}%
\end{equation}
 Denote by
$C_{s}^{\gamma,\gamma/2}(R_{+T}^{N})$ the space of functions
$u(x,t)$ with the finite norm

\begin{equation}
|u|_{C_{s}^{\gamma,\gamma/2}(R_{+T}^{N})}\equiv |u|^{(\gamma)}_{s,R^{N}_{+T}} \equiv
|u|_{R_{+T}^{N}}^{(0)}+H_{s,R_{+T}^{N}%
}^{\gamma}(u)+\left\langle u\right\rangle _{t,R_{+T}^{N}}^{(\frac{\gamma}{2}%
)}, \label{B4}%
\end{equation}
where $\left\langle u\right\rangle
_{t,R_{+T}^{N}}^{(\frac{\gamma}{2})}$ is the H\"{o}lder constant
with respect to $t$ with the exponent $\gamma/2$ of the function
 $u(x,t)$.

Difine further the space
$C_{s}^{2+\gamma,\frac{2+\gamma}{2}}(R_{+T}^{N})$ as the Banach
space of functions $u(x,t)$ with the finite norm

\[
|u|_{C_{s}^{2+\gamma,\frac{2+\gamma}{2}}(R_{+T}^{N})}\equiv
|u|^{(2+\gamma)}_{s,R^{N}_{+T}} \equiv
|u|_{C_{s}^{\gamma,\gamma/2}(R_{+T}^{N})}
+\sum\limits_{i=1}^{N}|u_{x_{i}}|_{C_{s}^{\gamma,\gamma/2}(R_{+T}^{N})}+
\]%

\begin{equation}
+|u_{t}%
|_{C_{s}^{\gamma,\gamma/2}(R_{+T}^{N})}+\sum\limits_{i,j=1}^{N}|x_{N}^{\alpha}u_{x_{i}x_{j}}|_{C_{s}^{\gamma,\gamma/2}%
(R_{+T}^{N})}. \label{B5}%
\end{equation}

Finally, the spaces
$C_{s}^{\gamma,\gamma/2}(\overline{\Omega}_{T}^{\pm})$ and
$C_{s}^{2+\gamma,\frac{2+\gamma}{2}}(\overline{\Omega}_{T}^{\pm})$
are defined as the spaces of functions $u(x,t)$ with the property,
that in some neighborhood of $\Gamma_{T}$ after the corresponding
change of variables  functions $u(x,t)$ belong to the space
$C_{s}^{\gamma,\gamma/2}(R_{+T}^{N})$ or to the space
$C_{s}^{2+\gamma,\frac{2+\gamma}{2}}(R_{+T}^{N})$ correspondingly,
and out of some neighborhood of $\Gamma_{T}$ the functions $u(x,t)$
belong to the standard spaces
$C^{\gamma,\gamma/2}(\overline{\Omega}_{T}^{\pm})$ or
$C^{2+\gamma,\frac{2+\gamma}{2}}(\overline{\Omega}_{T}^{\pm})$. In
particular, for a function $u\in
C_{s}^{2+\gamma,\frac{2+\gamma}{2}}(\overline{\Omega}_{T}^{\pm})$
the following norm is finite

\begin{equation}
\left|  (d^{\pm})^{\alpha}u_{x_{i}x_{j}}\right|  _{C_{s}^{\gamma,\gamma/2}%
(\overline{\Omega}_{T}^{\pm})}<\infty,\quad i,j=\overline{1,N}.\label{B6}%
\end{equation}

In the case of functions $u(x)$ from the variable $x$ only, $x\in
\overline{\Omega}^{\pm}$, the spaces
$C_{s}^{\gamma}(\overline{\Omega}^{\pm})$ and
$C_{s}^{2+\gamma}(\overline{\Omega}^{\pm})$ are defined in the
completely analogous way.

We will use also some standard anisotropic spaces of smooth
functions, which are more general than the spaces $H^{l,l/2}\equiv
C^{l,l/2}$. Namely, we will use the spaces

\[
C^{l_{1},l_{2}}(\overline{\Omega}^{\pm}_{T}), \ \
C^{l_{1},l_{2}}(\Gamma_{T}),
\]
where $l_{1}$, $l_{2}$ are noninteger positive numbers. Such spaces
are defined in \cite{15}, for example. They consist from the
functions, which have H\"{o}lder continuous with respect to $x$
withe the exponent $l_{1}-[l_{1}]$ derivatives with respect to $x$
up to the order $[l_{1}]$, and also these functions have H\"{o}lder
continuous with respect to $t$ withe the exponent $l_{2}-[l_{2}]$
derivatives with respect to $t$ up to the order $[l_{2}]$. The norm
of such space we denote by

\[
|u|_{C^{l_{1},l_{2}}(\overline{\Omega}^{\pm}_{T})}\equiv
|u|^{(l_{1},l_{2})}_{\overline{\Omega^{\pm}_{T}}}.
\]

In fact, the functions from such spaces possess the property, that
their derivatives with respect to $x$ are smooth with respect to $t$
and their derivatives with respect to $t$ are smooth with respect to
$x$. More precisely, if $k_{1}$ is a multiindex, that is,
$k_{1}=(k_{1,1}, k_{1,2},...,k_{1,N})$ where $k_{1,i}$ are
nonnegative integers, and $k_{2}$ is a nonnegative integer, then the
function $\partial^{k_{1}}_{x}\partial^{k_{2}}_{t}u(x,t)$ belongs to
the space $C^{m_{1},m_{2}}(\overline{\Omega}^{\pm}_{T})$ with the
exponents $m_{i}=\mu l_{i}$, where

\[
\mu=1-\frac{|k_{1}|}{l_{1}}-\frac{k_{2}}{l_{2}}
\]
(see \cite{15}), and also

\begin{equation}
|\partial^{k_{1}}_{x}\partial^{k_{2}}_{t}u|_{C^{m_{1},m_{2}}(\overline{\Omega}^{\pm}_{T})}\leq
C |u|_{C^{l_{1},l_{2}}(\overline{\Omega}^{\pm}_{T})}.
\label{Sol.1}
\end{equation}

In addition, we will use the spaces with zero at the bottom of their
notation, that is, the spaces (compare \cite{12}, Ch.IV)

\begin{equation}
C^{2+\gamma,\frac{2+\gamma}{2}}_{0,s}(\overline{\Omega}^{\pm}_{T}),
\ C_{0}^{l_{1},l_{2}}(\overline{\Omega}^{\pm}_{T}), \
C^{\gamma,\frac{\gamma}{2}}_{0,s}(\overline{\Omega}^{\pm}_{T}), \
C_{0}^{l_{1},l_{2}}(\Gamma_{T}). \label{B.2}
\end{equation}
Such notations means closed subspace of the corresponding space,
which consists of functions that vanish at $t=0$ together with all
their derivatives with respect to $t$ up to the highest possible
order in the corresponding space.

Let here and throughout below the exponent $\beta \in (0,1)$ is
connected to the exponent $\gamma$ by the equality

\begin{equation}
\beta=\gamma(1-\frac{\alpha}{2})\label{B7}%
\end{equation}
In particular, we will use the spaces
$C^{2+\beta-\alpha,1+\frac{\gamma}{2}}(\overline{\Omega}_{T}^{\pm})$
and $C^{2+\beta-\alpha,1+\frac{\gamma}{2}}(\Gamma_{T})$, which
consist of the functions $u(x,t)$ with smoothness with respect to
$x$ up to the order $2+\beta-\alpha$ and with smoothness with
respect to $t$ up to the order $1+\gamma/2$, that is

\begin{equation}
\sum_{i=1}^{N}\left\langle u_{x_{i}}\right\rangle
_{x,\overline{\Omega}_{T}^{\pm}}^{(1+\beta-\alpha
)}+\left\langle u_{t}\right\rangle _{t,\overline{\Omega}_{T}^{\pm}}^{(\frac{\gamma}{2})}<\infty.\label{B8}%
\end{equation}
It is easy to calculate, that because of the relation \eqref{B7}, we
have

\begin{equation}
1+\frac{\gamma}{2}=\frac{2+\beta-\alpha}{2-\alpha},\label{B9}%
\end{equation}
so we also use the following notation for the mentioned above spaces
\begin{equation}
C^{2+\beta-\alpha,\frac{2+\beta-\alpha}{2-\alpha}}(\overline{\Omega}_{T}%
^{\pm}),\quad
C^{2+\beta-\alpha,\frac{2+\beta-\alpha}{2-\alpha}}(\Gamma
_{T}).\label{B10}%
\end{equation}



Further, we suppose that the initial conditions in \eqref{1.7} are
such that

\begin{equation}
v_{0}^{\pm}(y)\equiv |u_{0}^{\pm}(y)|^{m-1}u_{0}^{\pm}(y)\in
C^{2+\gamma'}_{s}(\overline{\Omega^{\pm}}),
 \label{1.11}
\end{equation}
where $\gamma'>\gamma$. Besides, we suppose that

\begin{equation}
\frac{\partial v_{0}^{\pm}(y)}{\partial \overrightarrow{n}}\geq
\nu >0, \ \ y\in \Gamma,
 \label{1.12}
\end{equation}
where $\overrightarrow{n}$ is a normal vector to the surface
$\Gamma$ which is directed into $\Omega^{+}$.


We will show below, that the free (unknown) boundary $S_{T}$ in
\eqref{1.4}, \eqref{1.5}  can be parameterized in terms of its
deviation from the given surface $\Gamma_{T}=\Gamma \times [0,T]$.
We follow to \cite{16} to give the strict formulation. Let
$\omega=(\omega_{1},...,\omega_{N-1})$ is a local curvilinear
coordinates in a domain $\Theta$ on $\Gamma$. In some small enough
neighbourhood $\mathcal{N}$ in $R^{N}$ of the surface $\Gamma$ we
introduce the coordinates $(\omega,\lambda)$ in the way that for any
$x\in \mathcal{N}$ we have the following unique representation

\begin{equation}
x=x'(x)+\overrightarrow{n}(x'(x))\lambda\equiv x(\omega)+\overrightarrow{n}(\omega)\lambda, \label{ab1.1}
\end{equation}
where $x'(x)=x(\omega)$ is the point in the domain $\Theta$ on the
surface $\Gamma$ with the coordinates $\omega$,
$\overrightarrow{n}(\omega)$ - normal to $\Gamma$ at the point
$x(\omega)$ with the direction into $\Omega^{+}$, $\lambda\in R$
means, in fact, deviation of the point $x$ from $\Gamma$, at that
$\pm \lambda>0$ for $x\in \Omega^{\pm}$. We assume that the
mentioned above neighbourhood $\mathcal{N}$ of the surface $\Gamma$
is the set

\[
\mathcal{N}=\{x\in\Omega: |\lambda(x)|<\gamma_{0}\},
\]
where $\gamma_{0}$ is small enough and will be chosen below.

Let $\rho(x',t)\equiv\rho(\omega,t)$ is a small and regular
function, which is defined on the surface $\Gamma_{T}$. Let us note,
that here and in what follows we use the notation $\rho(\omega,t)$
with the argument $\omega$ instead of $\rho(x',t)$ for all functions
on the surface $\Gamma$ if it does not cause ambiguity. We do that
just for simplification of the notation, bearing in mind that in
each local domain $\Theta $ on $\Gamma $ we can introduce local
coordinates $\omega$. At the same time the coordinate $\lambda$ in
\eqref{ab1.1} does not depend on the choice of local coordinates
$\omega$.

We parameterize the unknown surface $S_{T}$ we the help of the
unknown function $\rho(\omega,t)$ as follows
\begin{equation}
S_{T}\equiv \Gamma_{\rho,T}=\{(x,t)\in \Omega_{T}: \
x=x'+\rho(x',t)\overrightarrow{n}(x')=x(\omega)+\rho(\omega,t)\overrightarrow{n}(\omega)\},
 \label{1.15}
\end{equation}
where $x'\equiv x(\omega)\in \Gamma$. Note, that this definition of
the surface $S_{T}\equiv \Gamma_{\rho,T}$ does not depend on a
choice of local coordinates $\omega$ in a particular local domain on
$\Gamma$. Thus, the unknown function $\rho(\omega,t)$  means, in
fact, deviation of the surface $\Gamma_{\rho,T}=S_{T}$ from the
given surface $\Gamma_{T}$.

Along with $Q^{+}_{T}$, $Q^{-}_{T}$ in \eqref{1.3} we use the
notation $\Omega^{+}_{\rho,T}=Q^{+}_{T}$ and
$\Omega^{-}_{\rho,T}=Q^{-}_{T}$ for the subdomains that
$\Gamma_{\rho,T}=S_{T}$ separates the domain $\Omega_{T}$. Let,
further, $ \rho(x,t)$ is an extension of the function
$\rho(\omega,t)$ from the surface $\Gamma_{T}$ to the whole domain
$\Omega_{T}$ to a function with the support in the neighborhood
$\mathcal{N}_{T}=\mathcal{N}\times [0, T]$ of the surface
$\Gamma_{T}$, $\rho(x,t)=E\rho(\omega,t)$, $E$ is some fixed
extension operator (the way of such an extension will be listed
below), at that we will denote $\rho^{\pm}\equiv
E\rho|_{\overline{\Omega}^{\pm}_{T}}\equiv E^{\pm}\rho $.

Define a mapping $e_{\rho}(x,t)$ from the domain
$\overline{\Omega_{T}}$ on itself with the help of the formula
$e_{\rho}: (x,t) \rightarrow (y,\tau)$, where, according to the
notations of \eqref{ab1.1},

\begin{equation}
y=\genfrac{\{}{.}{0pt}{}{x'(x)+\overrightarrow{n}(x'(x))(\lambda(x)+\rho(x,t)), \ \ x\in
\mathcal{N},}
{x, \ \ \ \ x\in \overline{\Omega}\setminus \mathcal{N},}%
\label{4.1dop}
\end{equation}
\[
\tau=t,
\]
or, with the help of the local coordinates $\omega$,

\begin{equation}
y=\genfrac{\{}{.}{0pt}{}{x'(\omega(x))+\overrightarrow{n}(\omega(x))(\lambda(x)+\rho(x,t)), \ \ x\in
\mathcal{N},}
{x, \ \ \ \ x\in \overline{\Omega}\setminus \mathcal{N},}%
\label{4.1}
\end{equation}
\[
\tau=t.
\]
Here $x'(x)\in \Gamma$, $\omega(x)$, $\lambda(x)$ are
$(\omega,\lambda)$- coordinates of a point $x$ in the neighbourhood
$\mathcal{N}$. Note here, that the definition of the mapping
$e_{\rho}$ does not depend on a choice of local coordinates $\omega$
 on the surface $\Gamma$.

We choose $\gamma_{0}$ small enough so that under the condition
\begin{equation}
|\rho|^{1+\beta}_{\Gamma}\leq 2\gamma_{0} \label{4.2}
\end{equation}
the mapping $e_{\rho}$ is a diffeomorphism of
$\overline{\Omega_{T}}$ on themselves, and also the mapping
$e_{\rho}$ is a diffeomorphism of the domains
$\overline{\Omega^{\pm}_{T}}$ on the domains
$\overline{\Omega^{\pm}_{\rho,T}}$. Note, that the surface
$\Gamma_{\rho,T}$ is exactly the image of the surface $\Gamma_{T}$
under this mapping, and the mapping $e_{\rho}(x,t)$ is the identical
mapping out of the neighbourhood $\mathcal{N}_{T}$ of $\Gamma_{T}$.


About the exponents of the H\"{o}lder spaces we use we suppose that

\begin{equation}
0<\gamma<\gamma'<1, \ \ \gamma<\min\{\alpha, 1-\alpha\}. \label{B.7}
\end{equation}

Note, that under our chice of $\gamma$ the restriction
$\gamma<\frac{\alpha}{1-\alpha/2}$ is also fulfilled. The last
restriction was introduced in \cite{12.1} at the studying of the
homogeneous Cauchy-Dirichlet problem for degenerate equations. We
need the restriction $\gamma<1-\alpha$ to have the inequality
$1+\beta-\alpha>1-\alpha>\gamma$, which implies that the first
derivatives with respect to $x$ of the functions from the classes
$C^{2+\beta-\alpha,\frac{2+\beta-\alpha}{2-\alpha}}$,
$C_{s}^{2+\gamma,\frac{2+\gamma}{2}}(\overline{\Omega}^{\pm}_{T})$
are more smooth than $C^{\gamma,\gamma/2}$,
$C^{\gamma,\gamma/2}_{s}$.  At last, we need the restriction
$\gamma<\alpha$ to expressions of the form
$(d^{\pm}(x))^{\alpha}\eta(x,t)$ with a smooth function $\eta(x,t)$
would belong to the corresponding space $C_{s}^{\gamma,\gamma/2}$.

Let us formulate now the main result.

\begin{theorem} \label{T1.1}
Let the conditions \eqref{1.8}-\eqref{1.10}, \eqref{B.7} on the data
of the problem \eqref{1.3} - \eqref{1.7} and the conditions
\eqref{1.11}, \eqref{1.12} are fulfilled. Let also the natural
adjustment conditions up to the first order at $\tau=0$, $y\in
\Gamma, \Gamma^{\pm}$  for the problem \eqref{1.3} - \eqref{1.7} are
fulfilled.  Then there exists such $T>0$, that on the time interval
$[0,T]$ the problem \eqref{1.3} - \eqref{1.7} has the unique smooth
solution, at that the the unknown boundary can be represented as in
\eqref{1.15} with the function $\rho(\omega,t)$ with the properties
\begin{equation}
\rho(\omega,t)\in
C^{2+\beta-\alpha,\frac{2+\beta-\alpha}{2-\alpha}}(\Gamma_{T}), \
\rho_{t}(\omega,t) \in
C^{1+\beta-\alpha,\frac{1+\beta-\alpha}{2-\alpha}}(\Gamma_{T}),
\label{T.1}
\end{equation}
\[
\rho^{\pm}(x,t)=E^{\pm}\rho(\omega,t)\in
C_{s}^{2+\gamma}(\overline{\Omega}^{\pm}_{T}),
\]
where $\rho(x,t)=E\rho(\omega,t)$ is the extension of the function
$\rho(\omega,t)$ to the domain $\overline{\Omega_{T}}$. The
functions $u^{\pm}(y,\tau)$ in \eqref{1.3} are such that
\begin{equation}
v^{\pm}(x,t)\equiv (|u^{\pm}(y,\tau)|^{m-1}u^{\pm}(y,\tau))\circ
e_{\rho}(x,t) \in C^{2+\gamma}_{s}(\overline{\Omega}^{\pm}_{T}).
\label{T.2}
\end{equation}
Thus, in particular, all of the relations of the problem \eqref{1.3}
- \eqref{1.7} are satisfied in the classical sense.
\end{theorem}

Subsequent sections of the paper devoted to the proof of the theorem
\ref{T1.1} in accordance with such a plan.

In the section \ref{s2} on the basis of  equivalent norms in the
spaces $C^{\gamma, \gamma/2}_{s}(\overline{\Omega}^{\pm}_{T})$ the
natural space of the traces on $\Gamma_{T}$ of the functions from
class $C^{2 + \gamma, \frac {2 +
\gamma}{2}}_{s}(\overline{\Omega}^{\pm}_{T})$ is studied. This
allows us to extend the results of \cite{12.1} about the solvability
of the homogeneous initial boundary value problem for a degenerate
equation to the case of the inhomogeneous problem.

These results are then used in the section \ref{s3} to study a model
Stefan problem for degenerate equations, which is one of the central
points of this paper. In this case, for the Schauder estimates of
the model problems the idea of the paper \cite{17} on the
application of the maximum principle to obtain the Schauder
estimates is used.

In the sections \ref{s4} and \ref{s5} the initial problem with the
unknown boundary is reduced to nonlinear problem in the fixed domain
and linearized on functions that expand  the initial data to the
domain $t>0$.

The section \ref{s6} is devoted to the study of the resulting linear
problem for degenerate equations on the basis of the results of the
section \ref{s3} about the properties of the corresponding model
problem. In this case, to prove the solvability of the linear
problem, we apply the idea of \cite{1} on parabolic regularization
of the Stefan boundary condition. Note that the corresponding model
problem in Section \ref{s3} is considered in the presence of the
regularization.

Finally, the section \ref{s7} completes the proof of the theorem
\ref{T1.1} by the method of \cite{2}.

\section{Auxiliary results on the spaces $C^{\gamma,\gamma/2}_{s}(\overline{\Omega}^{\pm}_{T})$,
$C^{2+\gamma,\frac{2+\gamma}{2}}_{s}(\overline{\Omega}^{\pm}_{T})$,
$C^{2+\beta-\alpha,\frac{2+\beta-\alpha}{2-\alpha}}(\Gamma_{T})$.}
\label{s2}

Note first, that for the spaces with zero in \eqref{B.2} the
following relations are valid. Let $\gamma'>\gamma$, $l_{1}'>l_{1}$,
$l_{2}'>l_{2}$. Let also functions $u$ and $v$ belong to one of the
mentioned spaces with the exponents of smoothness $\gamma'$,
$2+\gamma'$, $l_{1}'$, $l_{2}'$. Then

\begin{equation}
|u|^{(\gamma)}_{s,\overline{\Omega}^{\pm}_{T}}\leq
CT^{\mu}|u|^{(\gamma')}_{s,\overline{\Omega}^{\pm}_{T}},
\label{B.3}
\end{equation}

\begin{equation}
|u|^{(2+\gamma)}_{s,\overline{\Omega}^{\pm}_{T}}\leq
CT^{\mu}|u|^{(2+\gamma')}_{s,\overline{\Omega}^{\pm}_{T}},
\label{B.4}
\end{equation}

\begin{equation}
|uv|^{(\gamma)}_{s,\overline{\Omega}^{\pm}_{T}}\leq
CT^{\mu}|u|^{(\gamma)}_{s,\overline{\Omega}^{\pm}_{T}}
|v|^{(\gamma)}_{s,\overline{\Omega}^{\pm}_{T}}, \label{B.5}
\end{equation}

\begin{equation}
|u|^{(l_{1},l_{2})}_{\overline{\Omega}^{\pm}_{T}}\leq
CT^{\mu}|u|^{(l_{1}',l_{2}')}_{\overline{\Omega}^{\pm}_{T}},
\label{B.6}
\end{equation}
where $\mu$ is some positive constants which depend on  $\gamma$,
$\gamma'$, $l_{i}$, $l_{i}'$.

These inequalities are well known for the spaces
$C_{0}^{l_{1},l_{2}}$ (see \cite{12}, \cite{Bizh}), and for the
spaces $C^{\gamma,\gamma/2}_{0,s}$,
$C^{2+\gamma,\frac{2+\gamma}{2}}_{0,s}$ these inequalities are
completely analogous.


\subsection{An equivalent norm for the spaces  $C^{\gamma}_{s}(\overline{\Omega}^{\pm}_{T})$,
$C^{2+\gamma}_{s}(\overline{\Omega}^{\pm}_{T})$.} \label{ss2.1}

Along with the seminorm $H^{\gamma}_{s,R^{N}_{+T}}$ from \eqref{B3}
we consider in $R^{N}_{+T}$ the following weighted seminorm

\begin{equation}
H_{\alpha,R_{+T}^{N}}^{\gamma}(f)=\sup_{x,\overline{x}\in R_{+}^{N}}%
\frac{|f(x,t)-f(\overline{x},t)|}{|x-\overline{x}|^{\beta}}+\sup
_{x,\overline{x}\in R_{+}^{N}}\widetilde{x}_{N}^{\gamma\frac{\alpha}{2}}%
\frac{|f(x,t)-f(\overline{x},t)|}{|x-\overline{x}|^{\gamma}},\label{E2}%
\end{equation}
where $\widetilde{x}_{N}=\max\{x_{N},\overline{x}_{N}\}$, and here
and throughout we, without loss of generality, assume that
$\overline{x}_{N}\leq x_{N}$, so $\widetilde{x}_{N}=x_{N}$.

\begin{lemma}\label{LE1}
The seminorms $H_{\alpha,R_{+T}^{N}}^{\gamma}(f)$ and
$H_{s,R_{+T}^{N}}^{\gamma}(f)$ are equivalent.
\end{lemma}

\begin{proof}
Let the seminorm $H_{s,R_{+T}^{N}}^{\gamma}(f)$ is finite. We show,
that then

\begin{equation}
H_{\alpha,R_{+T}^{N}}^{\gamma}(f)\leq CH_{s,R_{+T}^{N}}^{\gamma}(f).\label{E3}%
\end{equation}

Let $\varepsilon_{0}\in (0,1)$ is small and fixed. Let
$x=(x',x_{N})$ and let first

\begin{equation}
|x^{\prime}-\overline{x}^{\prime}|\geq\varepsilon_{0}x_{N}.\label{E4}%
\end{equation}
Then the more

\begin{equation}
|x-\overline{x}|\geq|x^{\prime}-\overline{x}^{\prime}|\geq\varepsilon_{0}%
x_{N}.\label{E5}%
\end{equation}
Under this condition

\[
s(x,\overline{x})=\frac{|x-\overline{x}|}{x_{N}^{\alpha/2}+\overline{x}%
_{N}^{\alpha/2}+|x^{\prime}-\overline{x}^{\prime}|^{\alpha/2}}\leq
\]%

\[
\leq C\frac{|x-\overline{x}|}{|x-\overline{x}|^{\alpha/2}+\overline{x}%
_{N}^{\alpha/2}+|x^{\prime}-\overline{x}^{\prime}|^{\alpha/2}}\leq
C|x-\overline{x}|^{1-\frac{\alpha}{2}}.
\]
Therefore, as $\beta=\gamma(1-\frac{\alpha}{2})$,

\begin{equation}
\frac{|f(x,t)-f(\overline{x},t)|}{|x-\overline{x}|^{\beta}}\leq
C\frac
{|f(x,t)-f(\overline{x},t)|}{s(x,\overline{x})^{\gamma}}\leq CH_{s,R_{+T}^{N}%
}^{\gamma}(f).\label{E6}%
\end{equation}

Besides, because of  \eqref{E5}, and then of \eqref{E6},

\begin{equation}
x_{N}^{\gamma\frac{\alpha}{2}}\frac{|f(x,t)-f(\overline{x},t)|}{|x-\overline
{x}|^{\gamma}}\leq\frac{x_{N}^{\gamma\frac{\alpha}{2}}}{(\varepsilon_{0}%
x_{N})^{\gamma\frac{\alpha}{2}}}\frac{|f(x,t)-f(\overline{x},t)|}%
{|x-\overline{x}|^{\gamma(1-\alpha/2)}}\leq
CH_{s,R_{+T}^{N}}^{\gamma
}(f).\label{E7}%
\end{equation}

Let now

\begin{equation}
|x^{\prime}-\overline{x}^{\prime}|\leq\varepsilon_{0}x_{N}.\label{E8}%
\end{equation}
Under this condition, as it easy to see,

\begin{equation}
s(x,\overline{x})\sim Cx_{N}^{-\frac{\alpha}{2}}|x-\overline{x}|.\label{E9}%
\end{equation}
Consequently,

\begin{equation}
x_{N}^{\gamma\frac{\alpha}{2}}\frac{|f(x,t)-f(\overline{x},t)|}{|x-\overline
{x}|^{\gamma}}\leq
C\frac{|f(x,t)-f(\overline{x},t)|}{s(x,\overline
{x})^{\gamma}}\leq CH_{s,R_{+T}^{N}}^{\gamma}(f).\label{E10}%
\end{equation}

To estimate, further, the unweighted H\"{o}lder constant in the
definition of $H_{\alpha,R_{+T}^{N}}^{\gamma}(f)$, we consider the
two cases.

If

\[
|x_{N}-\overline{x}_{N}|\geq\varepsilon_{0}x_{N},
\]
then

\[
|x-\overline{x}|\geq|x_{N}-\overline{x}_{N}|\geq\varepsilon_{0}x_{N}%
\]
and therefore, as it was above,

\[
s(x,\overline{x})\leq\frac{|x-\overline{x}|}{(|x-\overline{x}|/\varepsilon
_{0})^{\alpha/2}}\leq C|x-\overline{x}|^{1-\frac{\alpha}{2}},
\]
so that, as above,

\begin{equation}
\frac{|f(x,t)-f(\overline{x},t)|}{|x-\overline{x}|^{\beta}}\leq
C\frac
{|f(x,t)-f(\overline{x},t)|}{s(x,\overline{x})^{\gamma}}\leq CH_{s,R_{+T}^{N}%
}^{\gamma}(f).\label{E11}%
\end{equation}

If now, under the condition \eqref{E8}, we have

\begin{equation}
|x_{N}-\overline{x}_{N}|\leq\varepsilon_{0}x_{N},\label{E12}%
\end{equation}
then in this case

\begin{equation}
|x-\overline{x}|\leq|x^{\prime}-\overline{x}^{\prime}|+|x_{N}-\overline{x}%
_{N}|\leq2\varepsilon_{0}x_{N}.\label{E13}%
\end{equation}
Therefore, in the force of \eqref{E9},

\[
s(x,\overline{x})\leq Cx_{N}^{-\alpha/2}|x-\overline{x}|\leq
\]%

\begin{equation}
\leq
Cx_{N}^{-\alpha/2}(2\varepsilon_{0}x_{N})^{\alpha/2}|x-\overline
{x}|^{1-\alpha/2}=C|x-\overline{x}|^{1-\alpha/2}.\label{E14}%
\end{equation}
Consequently, in this case

\begin{equation}
\frac{|f(x,t)-f(\overline{x},t)|}{|x-\overline{x}|^{\beta}}\leq
C\frac
{|f(x,t)-f(\overline{x},t)|}{s(x,\overline{x})^{\gamma}}\leq CH_{s,R_{+T}^{N}%
}^{\gamma}(f).\label{E15}%
\end{equation}
The estimate \eqref{E3} follows now from \eqref{E6}, \eqref{E7},
\eqref{E10}, \eqref{E11} and \eqref{E15}.

Further, let now the seminorm $H_{\alpha,R_{+T}^{N}}^{\gamma}(f)$ is
finite. Let us prove  the following estimate

\begin{equation}
H_{s,R_{+T}^{N}}^{\gamma}(f)\leq CH_{\alpha,R_{+T}^{N}}^{\gamma}(f).\label{E16}%
\end{equation}

Let first

\begin{equation}
|x^{\prime}-\overline{x}^{\prime}|\leq\varepsilon_{0}x_{N},\quad
x_{N}>0.\label{E17}%
\end{equation}
Then

\[
s(x,\overline{x})\geq\nu\frac{|x-\overline{x}|}{x_{N}^{\alpha/2}},
\]
and consequently

\begin{equation}
\frac{|f(x,t)-f(\overline{x},t)|}{s(x,\overline{x})^{\gamma}}\leq
Cx_{N}^{\gamma\alpha/2}\frac{|f(x,t)-f(\overline{x},t)|}{|x-\overline
{x}|^{\gamma}}\leq CH_{\alpha,R_{+T}^{N}}^{\gamma}(f).\label{E18}%
\end{equation}
In the particular case $x_{N}=0$ we have $\overline{x}_{N}=0$ and
therefore

\[
s(x,\overline{x})=|x^{\prime}-\overline{x}^{\prime}|^{1-\alpha/2}%
=|x-\overline{x}|^{1-\alpha/2},
\]
and so again

\begin{equation}
\frac{|f(x,t)-f(\overline{x},t)|}{s(x,\overline{x})^{\gamma}}=\frac
{|f(x,t)-f(\overline{x},t)|}{|x-\overline{x}|^{\beta}}\leq
CH_{\alpha
,R_{+T}^{N}}^{\gamma}(f).\label{E19}%
\end{equation}

Let now we have

\begin{equation}
|x^{\prime}-\overline{x}^{\prime}|\geq\varepsilon_{0}x_{N}.\label{E20}%
\end{equation}
Then

\begin{equation}
s(x,\overline{x})\geq\nu\frac{|x-\overline{x}|}{|x^{\prime}-\overline
{x}^{\prime}|^{\alpha/2}}\geq\nu|x-\overline{x}|^{1-\alpha/2},\label{E21}%
\end{equation}
and consequently,

\begin{equation}
\frac{|f(x,t)-f(\overline{x},t)|}{s(x,\overline{x})^{\gamma}}\leq
C\frac{|f(x,t)-f(\overline{x},t)|}{|x-\overline{x}|^{\beta}}\leq
CH_{\alpha,R_{+T}^{N}}^{\gamma}(f).\label{E22}%
\end{equation}

Thus, \eqref{E16} follows from  \eqref{E18}, \eqref{E19},
\eqref{E21}. And so the equivalence of the seminorms
$H_{s,R_{+T}^{N}}^{\gamma}(f)$ and
$H_{\alpha,R_{+T}^{N}}^{\gamma}(f)$ is proved.
\end{proof}

In this way, the norm in the space $C^{\gamma}_{s}(R^{N}_{+T})$ may
be given in the form

\begin{equation}
|u|_{C^{\gamma,\gamma/2}_{s}(R^{N}_{+T})}=|u|^{(0)}_{R^{N}_{+T}}+H_{\alpha,R_{+T}^{N}}^{\gamma}(u)
+\langle u\rangle^{\gamma/2}_{t,R^{N}_{+T}}.
\label{E23}%
\end{equation}

Bearing in mind the local straightening of the boundary $\Gamma$,
for the case of arbitrary domains $\overline{\Omega}_{T}^{\pm}$, the
norm in the space
$C^{\gamma,\gamma/2}_{s}(\overline{\Omega}_{T}^{\pm})$ may be
explicitly written as

\begin{equation}
|u|_{C^{\gamma}_{s}(\overline{\Omega}_{T}^{\pm})}=|u|^{(0)}_{\overline{\Omega}_{T}^{\pm}}+
\langle u
\rangle^{(\beta)}_{x,\overline{\Omega}_{T}^{\pm}}+\sup_{x,\overline{x}\in
\overline{\Omega}^{\pm}}(\widetilde{d^{\pm}(x,\overline{x})})^{\gamma
\alpha/2}\frac{|u(x,t)-u(\overline{x},t)|}{|x-\overline{x}|^{\gamma}}+
\langle u\rangle^{\gamma/2}_{t,\overline{\Omega}^{\pm}_{T}},
\label{E24}%
\end{equation}
where the functions $d^{\pm}(x)$ were introduced in the previous
section in \eqref{1.13} and they model the distance to the boundary
$\Gamma$, $\widetilde{d^{\pm}(x,\overline{x})}=\max
\{d^{\pm}(x),d^{\pm}(\overline{x})\}$.

Quite similar in terms of \eqref{E24} and \eqref{B6} we may
explicitly define the norm  \eqref{B5} in the space
$C^{2+\gamma}_{s}(\overline{\Omega}_{T}^{\pm})$.

\subsection{The traces of the functions from $C^{2+\gamma,\frac{2+\gamma}{2}}_{s}(\overline{\Omega}^{\pm}_{T})$
on $\Gamma_{T}$.}\label{ss2.2}

In view of the smoothness of the surface $\Gamma$, we can use local
straightening of the surface at consideration locally defined
classes
$C^{2+\gamma,\frac{2+\gamma}{2}}_{s}(\overline{\Omega}^{\pm}_{T})$.
Therefore it is sufficient to consider the case of the half-space,
that is to consider the finite in $R^{N}_{+T}=R^{N}_{+}\times [0,T]$
function $u(x,t)$ from the space
$C^{2+\gamma,\frac{2+\gamma}{2}}_{s}(R^{N}_{+T})$ and to consider
its trace at $x_{N}=0$.

\begin{lemma} \label{LSl1}
Let the function $u(x,t)$ is finite and $u(x,t)\in
C^{2+\gamma,\frac{2+\gamma}{2}}(R_{+T}^{N})$ , $0<\gamma<\alpha$,
$\beta=\gamma(1-\alpha/2)$. Then the function
 $v(x^{\prime},t)=u(x^{\prime},0,t)\in C^{2+\beta-\alpha,\frac
{2+\beta-\alpha}{2-\alpha}}(R_{T}^{N-1})=C^{2+\beta-\alpha,\frac{2+\gamma}{2}%
}(R_{T}^{N-1})$, at that

\begin{equation}
\left|  v(x^{\prime},t)\right|
_{C^{2+\beta-\alpha,\frac{2+\beta-\alpha
}{2-\alpha}}(R_{T}^{N-1})}=\left|  u(x^{\prime},0,t)\right|
_{C^{2+\beta
-\alpha,\frac{2+\beta-\alpha}{2-\alpha}}(R_{T}^{N-1})}\leq C\left|
u\right|
_{s,R_{+T}^{N}}^{(2+\gamma)}.\label{Sl.1}%
\end{equation}
Besides,

\begin{equation}
|\nabla_{(x',x_{N})} u(x^{\prime},0,t)|_{C^{1+\beta-\alpha,\frac{1+\beta-\alpha}{2-\alpha}%
}(R_{T}^{N-1})}\leq C_{T}|u|_{s,R_{+T}^{N}}^{(2+\gamma)}.\label{SlA.1}%
\end{equation}

\end{lemma}

\begin{proof}
It follows directly from the definition of the space
$C^{2+\gamma,\frac{2+\gamma}{2}}(R_{+T}^{N})$
 в \eqref{B5} and from the lemma  \ref{LE1}, that  $u_{t}(x^{\prime
},0,t)\in C^{\beta,\gamma/2}(R_{T}^{N-1})$, and therefore
$v_{t}(x^{\prime},0)=u_{t}(x^{\prime},0,t) \in C^{\beta,\gamma/2}
(R_{T}^{N-1})$, and in addition
\begin{equation}
\left|  v_{t}\right|  _{C^{\beta,\gamma/2}(R_{T}^{N-1})}\leq
C\left|
u\right|  _{s,R_{+T}^{N}}^{(2+\gamma)}.\label{Sl.2}%
\end{equation}
Therefore, in the force of \eqref{Sol.1} (see \cite{15}), it is
sufficient to prove uniformly in $t$ the following estimate

\begin{equation}
\left|  v(\cdot,t)\right|  _{R^{N-1}}^{(2+\beta-\alpha)}\leq
C\left|
u\right|  _{s,R_{+T}^{N}}^{(2+\gamma)},\label{Sl.3}%
\end{equation}
and for this it is sufficient to prove, that uniformly in $t$ and in
$x_{N}$ for all $i=\overline{1,N}$ for the function $w=u_{x_{i}}$ we
have

\begin{equation}
\left\langle w\right\rangle
_{x^{\prime},R^{N-1}}^{(1+\beta-\alpha)}\leq
C\left|  u\right|  _{s,R_{+T}^{N}}^{(2+\gamma)}\label{Sl.4}%
\end{equation}
So, let $w=u_{x_{i}}$, $i=\overline{1,N}$. To prove \eqref{Sl.4} it
is sufficient, as it follows from \cite{18}, to show that for
arbitrary $h>0$ the follows inequality holds

\begin{equation}
\frac{\left|  \Delta_{h,x^{\prime}}^{2}w(x_{N})\right|  }{h^{1+\beta-\alpha}%
}\leq C\left|  u\right|  _{s,R_{+T}^{N}}^{(2+\gamma)}.\label{Sl.5}%
\end{equation}
Here $\Delta_{h,x^{\prime}}^{2}w(x_{N})\equiv\Delta_{h,x^{\prime}}%
^{2}w\equiv\Delta_{h,x^{\prime}}^{2}w(x^{\prime},x_{N},t)$ is the
second difference from the function $w$ with respect to the variable
$x_{j}^{\prime}$, $j=\overline{1,N-1}$, with the step $h$, that is

\begin{equation}
\Delta_{h,x^{\prime}}^{2}w(x_{N})=u(x^{\prime}+\overrightarrow{j}%
h,x_{N},t)-2u(x^{\prime},x_{N},t)+u(x^{\prime}-\overrightarrow{j}%
h,x_{N},t).\label{Sl.6}%
\end{equation}

Consider the two cases.  Let first

\begin{equation}
h\leq x_{N}.\label{Sl.7}%
\end{equation}
Then according to the mean value theorem with some $\theta_{1},
\theta_{2}\in (0,1)$ we have

\[
\frac{\left|  \Delta_{h,x^{\prime}}^{2}w(x_{N})\right|  }{h^{1+\beta-\alpha}%
}\leq\frac{\left|
x_{N}^{\alpha}\Delta_{h,x^{\prime}}^{2}w(x_{N})\right|
}{h^{1+\beta}}=
\]%

\[
=\left|
x_{N}^{\alpha}\frac{w_{x_{j}}(x^{\prime}+\overrightarrow{j}\theta
_{1}h,x_{N},t)-w_{x_{j}}(x^{\prime}-\overrightarrow{j}\theta_{2}h,x_{N}%
,t)}{h^{\beta}}\right|  =
\]%

\[
=\frac{\left|
x_{N}^{\alpha}u_{x_{i}x_{j}}(x^{\prime}+\overrightarrow
{j}\theta_{1}h,x_{N},t)-x_{N}^{\alpha}u_{x_{i}x_{j}}(x^{\prime}%
+\overrightarrow{j}\theta_{2}h,x_{N},t)\right|  }{h^{\beta}}\leq
\]%

\begin{equation}
\leq C\left\langle x_{N}^{\alpha}u_{x_{i}x_{j}}\right\rangle
_{x^{\prime },R_{+T}^{N}}^{(\beta)}\leq C\left|  u\right|
_{s,R_{+T}^{N}}^{(2+\gamma
)}.\label{Sl.8}%
\end{equation}

Let now $h\geq x_{N}$. Write the difference
$\Delta_{h,x^{\prime}}^{2}w(x_{N})$ in the form

\[
\Delta_{h,x^{\prime}}^{2}w(x_{N})=-\Delta_{h,x^{\prime}}^{2}\left(
w(x^{\prime},x_{N}+h,t)-w(x^{\prime},x_{N}+h,t)\right)  +
\]%

\begin{equation}
+\Delta_{h,x^{\prime}}^{2}w(x^{\prime},x_{N}+h,t)\equiv
A_{1}+A_{2}.\label{Sl.9}
\end{equation}
In view of the fact that for the expression $A_{2}$ the condition
\eqref{Sl.7} holds, that is $h\leq x_{N}+h$, completely analogous to
\eqref{Sl.8},

\begin{equation}
\frac{\left|  A_{2}\right|  }{h^{1+\beta-\alpha}}\leq C\left|
u\right|
_{s,R_{+T}^{N}}^{(2+\gamma)}.\label{Sl.10}%
\end{equation}

To estimate the expression $|A_{1}|/h^{1+\beta-\alpha}$ we use the
formula

\[
w(x^{\prime},x_{N}+h,t)-w(x^{\prime},x_{N},t)=h\int\limits_{0}^{1}w_{x_{N}%
}(x^{\prime},x_{N}+\theta h,t)d\theta.
\]
Consequently, in view of $w_{x_{N}}=u_{x_{i}x_{N}}$,

\[
\frac{\left|  A_{1}\right|  }{h^{1+\beta-\alpha}}\leq\int\limits_{0}^{1}%
\frac{h^{\alpha}}{(x_{N}+\theta h)^{\alpha}}\left|
\frac{\Delta_{h,x^{\prime }}^{2}(x_{N}+\theta
h)^{\alpha}u_{x_{i}x_{N}}(x^{\prime},x_{N}+\theta
h,t)}{h^{\beta}}\right|  d\theta\leq
\]%

\begin{equation}
\leq C\left\langle x_{N}^{\alpha}u_{x_{i}x_{N}}\right\rangle
_{x^{\prime },R_{+T}^{N}}^{(\beta)}\int\limits_{0}^{1}\left(
x_{N}/h+\theta\right) ^{-\alpha}d\theta\leq C\left|  u\right|
_{s,R_{+T}^{N}}^{(2+\gamma
)}.\label{Sl.11}%
\end{equation}

Thus, from  \eqref{Sl.10} and \eqref{Sl.11} we obtain \eqref{Sl.5},
and so we have also \eqref{Sl.4} and \eqref{Sl.3}. Together with
\eqref{Sl.2} this completes the proof of \eqref{Sl.1}.

We now show the inequality \eqref{SlA.1}. Note that for tangential
derivatives $u_{x_{i}}$, $i=\overline{1,N-1}$ this inequality
follows from the above estimate \eqref{Sl.1} and from \cite{15},
\eqref{Sol.1}. However, for $u_{x_{N}}$ we need a separate proof. We
show \eqref{SlA.1} for $u_{x_{k}}$, $k=\overline{1,N}$.

According to \cite{18}, it is enough to show that for $h>0$

\begin{equation}
|\Delta_{h,t}^{2}u_{x_{k}}(x,t)|\leq
C|u|_{s,R_{+T}^{N}}^{(2+\gamma)}
h^{\frac{1+\beta-\alpha}{2-\alpha}},\label{SlA.2}
\end{equation}
where

\[
\Delta_{h,t}^{2}v(x,t)=\Delta_{h,t}^{2}v=v(x,t+2h)-2v(x,t+h)+v(x,t)
\]
is the second difference of a function $v$ with respect to the
variable $t$ with the step $h$.

Let first

\begin{equation}
h^{\frac{1}{2-\alpha}}\leq x_{N}.\label{SlA.2.1}%
\end{equation}
Then we use the following interpolation inequality (see, for
example, \cite{15}, \cite{18.1}, Ch.1)

\[
|v_{x_{k}}|_{\Pi_{T}(x_{N})}^{(0)}\leq C\left(  |v|_{\Pi_{T}(x_{N})}%
^{(0)}\right)  ^{1/2}\left(  |v|_{\Pi_{T}(x_{N})}^{(2)}\right)
^{1/2}.
\]
Here we denote $\Pi_{T}(x_{N})=\left\{
(y,t):x_{N}/2<y_{N}<3x_{N}/2,0<t<T\right\}$, and
$|v|_{\Pi_{T}(x_{N})}^{(2)}$ means $C^{2}$- norm with respect to
$x$- variables  over the specified domain. We obtain for
$\Delta_{h,t}^{2}u_{x_{k}}(x,t)$, that

\begin{equation}
|\Delta_{h,t}^{2}u_{x_{k}}(x,t)|\leq C\left(
|\Delta_{h,t}^{2}u|_{\Pi
_{T}(x_{N})}^{(0)}\right)  ^{1/2}\left(  |\Delta_{h,t}^{2}u|_{\Pi_{T}(x_{N}%
)}^{(2)}\right)  ^{1/2}.\label{SlA.3}%
\end{equation}
In view of the properties of the space
$C_{s}^{2+\gamma,\frac{2+\gamma}{2}}(R_{+T}^{N})$ (см. \cite{15},
\cite{18}, \eqref{Sol.1}), we have

\begin{equation}
|\Delta_{h,t}^{2}u|_{\Pi_{T}(x_{N})}^{(0)}\leq C|u|_{s,R_{+T}^{N}}%
^{(2+\gamma)}h^{\frac{2+\gamma}{2}},\label{SlA.4}%
\end{equation}%

\begin{equation}
|\Delta_{h,t}^{2}u|_{\Pi_{T}(x_{N})}^{(2)}\leq C|u|_{s,R_{+T}^{N}}%
^{(2+\gamma)}x_{N}^{-\alpha}h^{\frac{\gamma}{2}}.\label{SlA.5}%
\end{equation}
From \eqref{SlA.4} and \eqref{SlA.5} considering \eqref{SlA.2.1}, we
obtain

\begin{equation}
|\Delta_{h,t}^{2}u_{x_{k}}(x,t)|\leq C|u|_{s,R_{+T}^{N}}^{(2+\gamma)}%
h^{\frac{2+\gamma}{4}+\frac{\gamma}{4}-\frac{\alpha}{2(2-\alpha)}%
}=C|u|_{s,R_{+T}^{N}}^{(2+\gamma)}h^{\frac{1+\beta-\alpha}{(2-\alpha)}%
},\label{SlA.6}%
\end{equation}
that is the inequality \eqref{SlA.2}.

Let now

\begin{equation}
h^{\frac{1}{2-\alpha}}\geq x_{N}.\label{SlA.7}%
\end{equation}
Write $\Delta_{h,t}^{2}u_{x_{k}}(x,t)$ as

\[
\Delta_{h,t}^{2}u_{x_{k}}(x,t)=-\Delta_{h,t}^{2}\left[
u_{x_{k}}(x^{\prime
},x_{N}+h^{\frac{1}{2-\alpha}},t)-u_{x_{k}}(x,t)\right]  +
\]%

\begin{equation}
+\Delta_{h,t}^{2}u_{x_{k}}(x^{\prime},x_{N}+h^{\frac{1}{2-\alpha}},t)\equiv
A_{1}+A_{2},\label{SlA.8}%
\end{equation}
and for $A_{2}$, by the above case \eqref{SlA.2.1}, the estimate

\[
|A_{2}|\leq C|u|_{s,R_{+T}^{N}}^{(2+\gamma)}h^{\frac{1+\beta-\alpha}%
{(2-\alpha)}}
\]
is valid.

To estimate the expression $A_{1}$, write it as

\[
A_{1}=-h^{\frac{1}{2-\alpha}}\int\limits_{0}^{1}\Delta_{h,t}^{2}u_{x_{k}x_{N}%
}(x^{\prime},x_{N}+\theta h^{\frac{1}{2-\alpha}},t)d\theta.
\]
Thus, we have for $A_{1}$, that

\[
|A_{1}|\leq C|u|_{s,R_{+T}^{N}}^{(2+\gamma)}h^{\frac{1}{2-\alpha}}%
h^{\frac{\gamma}{2}}\int\limits_{0}^{1}\left(  x_{N}+\theta
h^{\frac {1}{2-\alpha}}\right)  ^{-\alpha}d\theta\leq
\]%

\[
\leq C|u|_{s,R_{+T}^{N}}^{(2+\gamma)}h^{\frac{1}{2-\alpha}+\frac{\gamma}%
{2}-\frac{\alpha}{2-\alpha}}\int\limits_{0}^{1}\theta^{-\alpha}d\theta
=C|u|_{s,R_{+T}^{N}}^{(2+\gamma)}h^{\frac{1+\beta-\alpha}{(2-\alpha)}},
\]
that is again the inequality \eqref{SlA.2}.

The lemma is proved.

 \end{proof}

Thus, due to the possibility of the local straightening of the
boundary, the following is true.

\begin{lemma}\label{LSl2}
Let functions $u^{\pm}(x,t)$ belong to the spaces
$C^{2+\gamma,\frac{2+\gamma}{2}}_{s}(\overline{\Omega}^{\pm}_{T})$.
Then the functions $v^{\pm}(x,t)=u^{\pm}(x,t)|_{x\in \Gamma}$
 belong to the space
$C^{2+\beta-\alpha,\frac{2+\beta-\alpha}{2-\alpha}}(\Gamma_{T})$,
and

\begin{equation}
|v^{\pm}|_{C^{2+\beta-\alpha,\frac{2+\beta-\alpha}{2-\alpha}}(\Gamma_{T})}
 \leq C\left|  u\right| _{s,\overline{\Omega}^{\pm}_{T}}^{(2+\gamma
)}.\label{Sl.12}%
\end{equation}
Besides,

\begin{equation}
|\nabla
u^{\pm}|_{\Gamma_{T}}|_{C^{1+\beta-\alpha,\frac{1+\beta-\alpha
}{2-\alpha}}(\Gamma_{T})}\leq C|u^{\pm}|_{s,\overline{\Omega}_{T}^{\pm}%
}^{(2+\gamma)}.\label{SlA.9}%
\end{equation}

\end{lemma}

\subsection{An extension from the surface $ \Gamma_{T}$ of the functions from the space
 $C^{2+\beta-\alpha,\frac{2+\beta-\alpha}{2-\alpha}}(\Gamma_{T})$. } \label{ss2.3}

In this section we prove the converse of Lemma \ref{LSl2}, that is,
we show that any function of the class
$C^{2+\beta-\alpha,\frac{2+\beta-\alpha}{2-\alpha}}(\Gamma_{T})=C^{2+\beta
-\alpha,1+\gamma/2}(\Gamma_{T})$ can be extended to all domains
$\overline{\Omega}_{T}^{\pm}$ can be extended to all region up to
functions of the class
$C^{2+\gamma,\frac{2+\gamma}{2}}_{s}(\overline{\Omega}_{T}^{\pm})$,
 and the extension operator is bounded (here, as above
$\beta=\gamma (1-\alpha/2)$). Such an extension operator is
constructed in the standard way by applying a sufficiently small
partition of the unity in the neighborhood of $\Gamma$ and by the
local straightening of the boundary $\Gamma$ - see \cite{12}. In
this case, it is enough to require the $H^{2+\gamma}$-smoothness of
the boundary $\Gamma$. Therefore, the existence of the said
extension operator follows in the standard way from the following
lemma.

\begin{lemma} \label{LP1}
Let in $R_{+T}^{N}$ at $x_{N}=0$ a finite function $f(x^{\prime},t)$
from the class $C^{2+\beta-\alpha,1+\gamma/2}(R_{T}^{N-1})$ is
given. Then $f$ can be extended in the domain  $x_{N}>0$ up to the
function  $u(x,t)$ from the class
$C_{s}^{2+\gamma,\frac{2+\gamma}{2}}(R_{+T}^{N})$, and

\begin{equation}
\left|  u\right|  _{s,R_{+T}^{N}}^{(2+\gamma)}\leq C\left|
f\right|
_{C^{2+\beta-\alpha,1+\gamma/2}(R_{T}^{N-1})}.\label{P.1}%
\end{equation}

\end{lemma}

\begin{proof}
Let $u(x,t)$ is the solution of the following Dirichlet problem with
the parameter $t\in [0,T]$:

\begin{equation}
\Delta u=0,\quad x\in R_{+}^{N}\quad(x_{N}>0),\label{P.2}%
\end{equation}%

\begin{equation}
u|_{x_{N}=0}=f(x^{\prime},t),\label{P.3}%
\end{equation}%

\begin{equation}
u\rightarrow0,\quad|x|\rightarrow\infty.\label{P.4}%
\end{equation}
As it is well known, the solution of \eqref{P.2}- \eqref{P.4} is
given by the potential of the double layer, defined by the Newton
potential.

Note, first, that for the problem \eqref{P.2}- \eqref{P.4} we have
the following maximum principle

\begin{equation}
|u|_{R_{+}^{N}}^{(0)}\leq|f|_{R^{N-1}}^{(0)}.\label{P.5}%
\end{equation}
Indeed, by \eqref{P.4}, we can choose $K>0$ so large that $|u| \leq
|f|^{(0)}_{R^{N-1}_{T}}/2$ for $|x| \geq K $, and, by the properties
of the double layer potential and the finiteness of $f$, a $K$ can
be chosen independent of $t$. Now consider in the domain $B_{K} =
R_{+}^{N} \cap \left\{|x|<K \right\}$ the functions $v^{\pm} = \pm u
+ |f|_{R^{N-1}}^{(0)} $. Within this domain we have

\begin{equation}
\Delta v^{\pm}=0,\quad x\in B_{K},\label{P.6}%
\end{equation}
and on the boundary $\partial B_{K}=\left\{ x_{N}=0,|x^{\prime}|\leq
K\right\} \cup\left\{ x_{N}>0,|x^{\prime}|=K\right\}$ the inequality

\begin{equation}
v^{\pm}\geq 0,\quad x\in\partial B_{K}\label{P.7}%
\end{equation}
holds. It follows from \eqref{P.6}, \eqref{P.7} and from the maximum
principle, that $v^{\pm}\geq 0$ in $\overline{B}_{K}$ for all $t$,
and, thus,

\[
|u|_{B_{K}}^{(0)}\leq|f|_{R^{N-1}}^{(0)} \ \ t\in[0,T].
\]
Due to the choice of $K$, we have the inequality \eqref{P.5} on
whole domain $R^{N}_{+T}$.

Consider first the properties of the function $u(x,t)$ with respect
to $t$. Let $v(x,t)$ is the solution of \eqref{P.2}- \eqref{P.4}
with the boundary condition

\begin{equation}
v|_{x_{N}=0}=f_{t}(x^{\prime},t)\label{P.8}%
\end{equation}
instead of \eqref{P.3}.  Consider also for $h>0$ the following
function
\[
u_{h}(x,t)=\frac{u(x,t+h)-u(x,t)}{h},
\]
which satisfies the problem \eqref{P.2}- \eqref{P.4} with the
boundary condition

\begin{equation}
u_{h}|_{x_{N}=0}=f_{h}(x^{\prime},t)\equiv\frac{f(x^{\prime},t+h)-f(x^{\prime
},t)}{h}.\label{P.9}%
\end{equation}
Let further $w(x,t)=u_{h}(x,t)-v(x,t)$, and the function $w(x,t)$
also satisfies the problem \eqref{P.2}- \eqref{P.4}, but with the
boundary condition

\begin{equation}
w|_{x_{N}=0}=f_{h}(x^{\prime},t)-f_{t}(x^{\prime},t)\equiv\varphi
_{h}(x^{\prime},t).\label{P.10}%
\end{equation}
Due to the properties of the function $f(x^{\prime},t)$, we have
with some $\theta(x^{\prime},t,h)\in(0,1)$ according to the mean
value theorem

\[
|\varphi_{h}(x^{\prime},t)|=|f_{h}(x^{\prime},t)-f_{t}(x^{\prime},t)|=
\]%

\begin{equation}
=|f_{t}(x^{\prime},t+\theta
h)-f_{t}(x^{\prime},t)|\leq\left\langle
f_{t}(x^{\prime},t)\right\rangle
_{t,R_{T}^{N-1}}^{(\gamma/2)}h^{\gamma
/2}\rightarrow0,\quad h\rightarrow0.\label{P.11}%
\end{equation}
Consequently, on the base of \eqref{P.5},

\begin{equation}
|u_{h}-v|_{R_{+T}^{N}}^{(0)}\leq Ch^{\gamma/2}\rightarrow0,\quad
h\rightarrow0.\label{P.12}%
\end{equation}
This means, that the function $u(x,t)$ has the derivative with
respect to the variable $t$ for $x\in R^{N}_{+T}$, and
$u_{t}(x,t)=v(x,t)$, that is $u_{t}(x,t)$ satisfies the problem
\eqref{P.2}- \eqref{P.4} with the boundary condition \eqref{P.8}.

Further, considering the function

\[
v_{h}(x,t)=\frac{u_{t}(x,t+h)-u_{t}(x,t)}{h^{\gamma/2}},
\]
we see, that it satisfies the same problem with the boundary
condition

\begin{equation}
v_{h}(x,t)|_{x_{N}=0}=\frac{f_{t}(x^{\prime},t+h)-f_{t}(x^{\prime}%
,t)}{h^{\gamma/2}}\equiv f_{th}(x^{\prime},t).\label{P.13}%
\end{equation}
Thus, on the base of \eqref{P.5} again,

\begin{equation}
\left|  \frac{u_{t}(x,t+h)-u_{t}(x,t)}{h^{\gamma/2}}\right|  _{R_{+T}^{N}%
}^{(0)}\leq C\left|  \frac{f_{t}(x^{\prime},t+h)-f_{t}(x^{\prime}%
,t)}{h^{\gamma/2}}\right|  _{R_{T}^{N-1}}^{(0)},\label{P.14}%
\end{equation}
which, by the arbitrariness of $h$, means that

\begin{equation}
\left\langle u_{t}(x,t)\right\rangle
_{t,R_{+T}^{N}}^{(\gamma/2)}\leq C\left\langle
f_{t}(x^{\prime},t)\right\rangle _{t,R_{T}^{N-1}}^{(\gamma
/2)}.\label{P.15}%
\end{equation}

Consider now the properties of the function $u(x,t)$ with respect to
the variables $x$.

First, it follows from the results of \cite{19}, \cite{20}, that for
each $t\in [0,T]$, due to
 $f\in C_{x}^{2+\beta-\alpha}(R_{T}^{N-1})$, we have $u\in C_{x}
^{2+\beta-\alpha}(R_{+T}^{N})$, and
\begin{equation}
|u|_{C_{x}^{2+\beta-\alpha}(R_{+T}^{N})}\leq
C|f|_{C_{x}^{2+\beta-\alpha
}(R_{T}^{N-1})},\label{P.16}%
\end{equation}
where the symbol $x$ at the bottom of the space notation means that
we consider the smoothness only with respect to $x$.

Show that the following estimates

\begin{equation}
\left\langle x_{N}^{\alpha}D^{2}u\right\rangle
_{x,R_{+T}^{N}}^{(\beta)}\leq
C|f|_{C_{x}^{2+\beta-\alpha}(R_{T}^{N-1})},\label{P.17}%
\end{equation}%

\begin{equation}
\left\langle \widetilde{x}_{N}^{\gamma\alpha/2}\left(  x_{N}^{\alpha}%
D^{2}u\right)  \right\rangle _{x,R_{+T}^{N}}^{(\gamma)}\leq C|f|_{C_{x}%
^{2+\beta-\alpha}(R_{T}^{N-1})}\label{P.17.1}
\end{equation}
are valid, that is

\begin{equation}
H_{\alpha}^{\gamma}(x_{N}^{\alpha}D^{2}u)\leq
C|f|_{C_{x}^{2+\beta-\alpha }(R_{T}^{N-1})}.\label{P.17.2}
\end{equation}

We will use the fact that, as it follows from \cite{21}, Ch.5.4, the
condition $f \in C_{x}^{l}(R^{N-1})$ in \eqref{P.3}, \ $l \in
(0,2)$, is equivalent to the condition

\begin{equation}
|D_{x}^{k}u|\leq C_{k}x_{N}^{-k+l}|f|_{C_{x}^{l}(R^{N-1})},\quad
k\geq 2,\label{P.18}%
\end{equation}
where here and below $D_{x}^{k}u=D^{k}u$ means a derivative of the
$k$-th order with respect to $x$ of the function $u(x,t)$.

Since it is important to prove \eqref{P.17} for $x_{N}<1$ only (for
$x_{N}>1$ such the estimate follows from the local estimates and is
well- known), we consider only the case $x_{N}<1$.

We also use the well-known interpolation inequality

\begin{equation}
\left\langle v(x)\right\rangle
_{x,\overline{\Omega}}^{(\beta)}\leq C\left(
|v|_{\overline{\Omega}}^{(0)}\right)  ^{1-\beta}\left(
|v|_{\overline{\Omega
}}^{(1)}\right)  ^{\beta},\label{P.19}%
\end{equation}
which is valid for the functions $v(x)\in C^{1}(\overline{\Omega})$,
$\Omega$ is a (possibly unbounded) domain with the sufficiently
smooth boundary (see, for example, \cite{18.1}, Ch.1 ).

In addition, at the proof of \eqref{P.17}, without loss of
generality, we prove smoothness of the function
$x_{N}^{\alpha}D^{2}u$ with respect to $x^{\prime}$ and with respect
to $x_{N}$ separately and we obtain the estimate \eqref{P.17}
separately for these two cases.

So, let first $x_{N}$ is fixed. Then, by \eqref{P.19} and
\eqref{P.18},

\[
\left\langle x_{N}^{\alpha}D^{2}u\right\rangle
_{x^{\prime}}^{(\beta)}\leq
C\left(  |x_{N}^{\alpha}D^{2}u|^{(0)}\right)  ^{1-\beta}\left(  |x_{N}%
^{\alpha}D^{2}u|^{(0)}+|\nabla_{x^{\prime}}x_{N}^{\alpha}D^{2}u|^{(0)}\right)
^{\beta}\leq
\]%

\begin{equation}
\leq C|f|_{C_{x}^{2+\beta-\alpha}(R_{T}^{N-1})}\left(  x_{N}^{\alpha}%
x_{N}^{-2+(2+\beta-\alpha)}\right)  ^{1-\beta}\left(  x_{N}^{\alpha}%
x_{N}^{-3+(2+\beta-\alpha)}\right)
^{\beta}=C|f|_{C_{x}^{2+\beta-\alpha
}(R_{T}^{N-1})},\label{P.20}%
\end{equation}
that is estimate \eqref{P.17} with respect to $x'$.

Analogously, using \eqref{P.19} and \eqref{P.18}, we prove
\eqref{P.17.1} with respect to $x'$:

\[
x_{N}^{\gamma\alpha/2}\left\langle x_{N}^{\alpha}D^{2}u\right\rangle
_{x^{\prime}}^{(\gamma)}\leq Cx_{N}^{\gamma\alpha/2}\left(
|x_{N}^{\alpha
}D^{2}u|^{(0)}\right)  ^{1-\gamma}\left(  |x_{N}^{\alpha}D^{2}u|_{x^{\prime}%
}^{(1)}\right)  ^{\gamma}\leq
\]%

\[
\leq
C|f|_{C_{x}^{2+\beta-\alpha}(R_{T}^{N-1})}x_{N}^{\gamma\alpha/2}\left(
x_{N}^{\alpha}x_{N}^{-2+(2+\beta-\alpha)}\right)
^{1-\gamma}\left( x_{N}^{\alpha}x_{N}^{-3+(2+\beta-\alpha)}\right)
^{\gamma}=
\]%

\begin{equation}
=C|f|_{C_{x}^{2+\beta-\alpha}(R_{T}^{N-1})}x_{N}^{\gamma\alpha/2+\beta-\gamma
}=C|f|_{C_{x}^{2+\beta-\alpha}(R_{T}^{N-1})}.\label{P.21}%
\end{equation}

We prove now the relations \eqref{P.17}, \eqref{P.17.1} with respect
to the variable $x_{N}$. For this we fix some $\varepsilon_{0}\in
(0,1/16)$ and consider the two cases, assuming without loss of
generality that $\overline{x}_{N}\leq x_{N}$.

Let first

\begin{equation}
|x_{N}-\ \overline{x}_{N}|=(x_{N}-\ \overline{x}_{N})\geq\varepsilon_{0}%
x_{N}.\label{P.22}%
\end{equation}
Then

\begin{equation}
\frac{|x_{N}^{\alpha}D^{2}u(x,t)-\overline{x}_{N}^{\alpha}D^{2}u(\overline
{x},t)|}{|x_{N}-\ \overline{x}_{N}|^{\beta}}\leq C\left(
|x_{N}^{\alpha
-\beta}D^{2}u(x,t)|+|\overline{x}_{N}^{\alpha-\beta}D^{2}u(\overline
{x},t)|\right)  .\label{P.23}%
\end{equation}
In this case, as above

\begin{equation}
|x_{N}^{\alpha-\beta}D^{2}u(x,t)|\leq C|f|_{C_{x}^{2+\beta-\alpha}(R_{T}%
^{N-1})}x_{N}^{\alpha-\beta}x_{N}^{-2+(2+\beta-\alpha)}=C|f|_{C_{x}%
^{2+\beta-\alpha}(R_{T}^{N-1})},\label{P.24}%
\end{equation}
and similarly for
$|\overline{x}_{N}^{\alpha-\beta}D^{2}u(\overline{x},t)|$.

In the same way

\begin{equation}
x_{N}^{\gamma\alpha/2}\frac{|x_{N}^{\alpha}D^{2}u(x,t)-\overline{x}%
_{N}^{\alpha}D^{2}u(\overline{x},t)|}{|x_{N}-\
\overline{x}_{N}|^{\gamma}}\leq C\left(
|x_{N}^{\alpha-\beta}D^{2}u(x,t)|+|\overline{x}_{N}^{\alpha-\beta
}D^{2}u(\overline{x},t)|\right)  \label{P.25}%
\end{equation}
and then proceeding as in  \eqref{P.24}.

Let now

\begin{equation}
0<(x_{N}-\overline{x}_{N})\leq\varepsilon_{0}x_{N},\label{P.26}%
\end{equation}
and let also

\[
\Pi(x_{N})=\left\{  y\in R_{+}^{N}:x_{N}-2\varepsilon_{0}x_{N}\leq
y_{N}\leq x_{N}+2\varepsilon_{0}x_{N}\right\}  ,
\]%

\begin{equation}
\Pi_{T}(x_{N})=\Pi(x_{N})\times\lbrack0,T].\label{P.27}%
\end{equation}
Then, taking into account that on $\Pi_{T}(x_{N})$ we have
$y_{N}\sim x_{N}$, as in the previous case

\[
\frac{|x_{N}^{\alpha}D^{2}u(x,t)-\overline{x}_{N}^{\alpha}D^{2}u(\overline
{x},t)|}{|x_{N}-\ \overline{x}_{N}|^{\beta}}\leq\left\langle
y_{N}^{\alpha }D^{2}u(y,t)\right\rangle
_{y,\Pi_{T}(x_{N})}^{(\beta)}\leq
\]%

\begin{equation}
\leq C\left(  |y_{N}^{\alpha-\beta}D^{2}u|_{\Pi_{T}(x_{N})}^{(0)}%
+x_{N}^{\alpha}\left\langle D^{2}u(y,t)\right\rangle _{y,\Pi_{T}(x_{N}%
)}^{(\beta)}\right)  \equiv A_{1}+A_{2}.\label{P.28}%
\end{equation}
Here $A_{1}$ is estimated in the same way as in \eqref{P.24}, and
$A_{2}$ - as well as in \eqref{P.20}, which gives

\begin{equation}
\frac{|x_{N}^{\alpha}D^{2}u(x,t)-\overline{x}_{N}^{\alpha}D^{2}u(\overline
{x},t)|}{|x_{N}-\ \overline{x}_{N}|^{\beta}}\leq
C|f|_{C_{x}^{2+\beta-\alpha
}(R_{T}^{N-1})}.\label{P.29}%
\end{equation}

The estimate

\begin{equation}
x_{N}^{\gamma\alpha/2}\frac{|x_{N}^{\alpha}D^{2}u(x,t)-\overline{x}%
_{N}^{\alpha}D^{2}u(\overline{x},t)|}{|x_{N}-\
\overline{x}_{N}|^{\gamma}}\leq
C|f|_{C_{x}^{2+\beta-\alpha}(R_{T}^{N-1})}\label{P.30}%
\end{equation}
is quite similar.

This completes the proof of \eqref{P.17.2}.

Similarly, we obtain the properties with respect to the variables
$x$ of the derivative $u_{t}$, that is,

\begin{equation}
H_{\alpha}^{\gamma}(u_{t})\leq C|f|_{C_{x}^{2+\beta-\alpha}(R_{T}^{N-1}%
)}\label{P.31}%
\end{equation}
because $u_{t}$ satisfies the problem \eqref{P.2}- \eqref{P.4} with
the boundary condition \eqref{P.8}.

Indeed, since  $u_{t}|_{x_{N}=0}=f_{t}$, so

\begin{equation}
\left\langle u_{t}\right\rangle _{x,R_{+T}^{N}}^{(\beta)}\leq
C\left\langle f_{t}\right\rangle _{x,R_{T}^{N-1}}^{(\beta)}\leq
C|f|_{C^{2+\beta
-\alpha,\frac{2+\beta-\alpha}{2-\alpha}}(R_{T}^{N-1})}.\label{P.32}%
\end{equation}

Further, for $x,\overline{x}\in R_{+}^{N}$,
$x_{N}\geq\overline{x}_{N}$ consider the difference

\begin{equation}
\Delta(x,\overline{x})u_{t}=x_{N}^{\gamma\alpha/2}\frac{|u_{t}(x,t)-u_{t}%
(\overline{x},t)|}{|x-\ \overline{x}|^{\gamma}}.\label{P.33}%
\end{equation}
If $|x-\ \overline{x}|\geq\varepsilon_{0} x_{N}$, then

\[
\Delta(x,\overline{x})u_{t}\leq C\frac{x_{N}^{\gamma\alpha/2}}{x_{N}%
^{\gamma\alpha/2}}\frac{|u_{t}(x,t)-u_{t}(\overline{x},t)|}{|x-\
\overline {x}|^{\beta}}\leq C\left\langle u_{t}\right\rangle
_{x,R_{+T}^{N}}^{(\beta )}\leq
\]%

\begin{equation}
\leq C|f|_{C^{2+\beta-\alpha,\frac{2+\beta-\alpha}{2-\alpha}}(R_{T}^{N-1}%
)}.\label{P.34}%
\end{equation}

If now  $|x-\ \overline{x}|\leq \varepsilon_{0} x_{N}$, then $x_{N}
\sim\overline{x}_{N}$, and then, using \eqref{P.18}, we obtain that

\[
\Delta(x,\overline{x})u_{t}\leq Cx_{N}^{\gamma\alpha/2}\frac{|u_{t}%
(x,t)-u_{t}(\overline{x},t)|}{|x-\ \overline{x}|}|x-\
\overline{x}|^{1-\gamma }\leq
\]%

\begin{equation}
\leq Cx_{N}^{\gamma\alpha/2}|\nabla_{x}u_{t}(x,t)|_{\Pi_{T}(x_{N})}^{(0)}%
x_{N}^{1-\gamma}\leq
C|f_{t}|_{C_{x}^{\beta}(R_{T}^{N-1})}x_{N}^{\gamma
\alpha/2+1-\gamma}x_{N}^{-1+\beta}=C|f_{t}|_{C_{x}^{\beta}(R_{T}^{N-1}%
)}.\label{P.35}%
\end{equation}

Now \eqref{P.31} follows from  \eqref{P.34} and \eqref{P.35}.

Let us show now the smoothness of the function
$x_{N}^{\alpha}D^{2}u(x,t)$ with respect to the variable $t$, that
is show that

\begin{equation}
\left\langle x_{N}^{\alpha}D^{2}u\right\rangle
_{t,R_{+T}^{N}}^{(\gamma
/2)}\leq C|f|_{C^{2+\beta-\alpha,\frac{2+\beta-\alpha}{2-\alpha}}(R_{T}%
^{N-1})}=C|f|_{C^{2+\beta-\alpha,1+\gamma/2}(R_{T}^{N-1})}.\label{P.36}%
\end{equation}

For this we fix some $h>0$ and consider the function

\begin{equation}
v_{h}(x,t)=\frac{u(x,t+h)-u(x,t)}{h^{\gamma/2}},\label{P.37}%
\end{equation}
which satisfies the problem \eqref{P.2}- \eqref{P.4} with the
following boundary condition

\begin{equation}
v_{h}(x,t)|_{x_{N}=0}=\varphi_{h}(x,t)\equiv\frac{f(x,t+h)-f(x,t)}%
{h^{\gamma/2}}.\label{P.38}%
\end{equation}
It follows from the results of \cite{15}, \eqref{Sol.1}, that
uniformly with respect to the variable $t$  the function
$\varphi_{h}(x,t)\in C_{x}^{2-\alpha}(R^{N-1})$ with respect to the
variables $x$, and

\begin{equation}
\max_{t\in\lbrack0,T]}\left|  \varphi_{h}(\cdot,t)\right|  _{R^{N-1}%
}^{(2-\alpha)}\leq C|f|_{C^{2+\beta-\alpha,1+\gamma/2}(R_{T}^{N-1}%
)}.\label{P.39}%
\end{equation}

Note now, that

\[
\frac{x_{N}^{\alpha}D^{2}u(x,t+h)-x_{N}^{\alpha}D^{2}u(x,t)}{h^{\gamma/2}%
}=x_{N}^{\alpha}D^{2}v_{h}(x,t).
\]

Consequently, it follows from \eqref{P.18} that

\begin{equation}
|x_{N}^{\alpha}D^{2}v_{h}(x,t)|\leq
Cx_{N}^{\alpha}\max_{t\in\lbrack
0,T]}\left|  \varphi_{h}(\cdot,t)\right|  _{R^{N-1}}^{(2-\alpha)}%
x_{N}^{-2+(2-\alpha)}=C\max_{t\in\lbrack0,T]}\left|  \varphi_{h}%
(\cdot,t)\right|  _{R^{N-1}}^{(2-\alpha)}.\label{P.49}%
\end{equation}

So, \eqref{P.36} follows from \eqref{P.49} and \eqref{P.39}, in view
of the definition of $v(x,t)$.

Multiplying now the  function $u(x, t)$ by a smooth finite function
$\eta(x)$, which is equal to one in a neighborhood of support of
$f(x',t)$, we get a finite extension of $f(x',t)$ of desired class,
and the estimate \eqref{P.1}.

The lemma \ref{LP1} is proved.
\end{proof}

From this lemma in the standard way (see \cite{12}), as it was
described in the beginning of this section, we get the following
assertion.

\begin{lemma}\label{LP2}
There exist bounded extension operators $E^{+}$ and $E^{-}$, such
that

\begin{equation} \rho\in
C^{2+\beta-\alpha,\frac{2+\beta-\alpha}{2-\alpha}}(\Gamma_{T})\rightarrow
E\rho \equiv E^{\pm}\rho\equiv \rho^{\pm}\in
C^{2+\gamma,\frac{2+\gamma}{2}}_{s}(\overline{\Omega}^{\pm}_{T}),
\label{P.50}
\end{equation}

\begin{equation}
|E^{\pm}\rho|_{C^{2+\gamma,\frac{2+\gamma}{2}}_{s}(\overline{\Omega}^{\pm}_{T})}
 \leq C
|\rho|_{C^{2+\beta-\alpha,\frac{2+\beta-\alpha}{2-\alpha}}(\Gamma_{T})},
 \label{P.51}
\end{equation}
and we can assume that the supports of the extended functions
$\rho^{\pm}$ are included in the sufficiently small neighbourhood
$\mathcal{N}_{T}$ of the surface $\Gamma_{T}$. We will denote the
extended functions $\rho\equiv E \rho\equiv E^{\pm}\rho$ by the same
symbol $\rho$ to not to overload the notation, that is,

\begin{equation}
\rho|_{\overline{\Omega}^{\pm}_{T}} \equiv  \rho^{\pm}\equiv
E^{\pm}\rho \equiv E\rho|_{\overline{\Omega}^{\pm}_{T}} .
 \label{P.510}
\end{equation}

\end{lemma}

Besides, as it follows from the results of \cite{12.1} and from the
lemma \ref{LP1} (as the lemma \ref{LP1} permits to reduce the
situation to the homogeneous boundary conditions), the following
assertion is valid.

\begin{lemma} \label{LPA1}
Let functions $f$ and $g$ are finite, and

\[
f(x^{\prime},t)\in C_{0}^{2+\beta-\alpha,\frac{2+\beta-\alpha}{2-\alpha}%
}(R_{T}^{N-1}),\quad g(x,t)\in C_{s.0}^{\gamma}(R_{+T}^{N}).
\]

Then the problem

\begin{equation}
\frac{\partial u}{\partial t}-x_{N}^{\alpha}\Delta
u=g(x,t),\quad(x,t)\in
R_{+T}^{N},\label{P.A.1}%
\end{equation}%

\begin{equation}
u(x^{\prime},0,t)=f(x^{\prime},t),\quad x_{N}=0,t\in\lbrack0,T],\label{P.A.2}%
\end{equation}%

\begin{equation}
u(x,0)=0,\quad x\in R_{+}^{N}\label{P.A.3}%
\end{equation}
has the unique solution $u(x,t)$, which satisfies the estimate

\begin{equation}
|u|_{s,R_{+T}^{N}}^{(2+\gamma)}\leq C\left(
|f|_{C^{2+\beta-\alpha
,\frac{2+\beta-\alpha}{2-\alpha}}(R_{T}^{N-1})}+|g|_{s,R_{+T}^{N}}^{(\gamma
)}\right)  .\label{P.A.4}%
\end{equation}

\end{lemma}


In the same way, with the help of results of \cite{12.1} and from
the lemma \ref{LP1} we get the following theorem.

Consider the Cauchy-Dirichlet problem for the degenerate equations
of the form

\begin{equation}
\frac{\partial u^{\pm}}{\partial t}-\left(  d^{\pm}(x)\right)  ^{\alpha}%
B^{\pm}(x,t)\Delta u^{\pm}=g^{\pm}(x,t),\quad(x,t)\in\Omega_{T}^{\pm},\label{B.8}%
\end{equation}%

\begin{equation}
u^{\pm}|_{\Gamma_{T}}=f^{\pm}(x,t),\label{B.9}%
\end{equation}%

\begin{equation}
u^{\pm}|_{\Gamma_{T}^{\pm}}=h^{\pm}(x,t),\label{B.10}%
\end{equation}%

\begin{equation}
u^{\pm}(x,0)=0,\label{B.11}%
\end{equation}
where the functions $d^{\pm}(x)$ are introduced in \eqref{1.13},

\[
B^{\pm}(x,t)\in
C^{\gamma,\gamma/2}(\overline{\Omega}_{T}^{\pm}),\quad\nu\leq
B^{\pm}\leq\nu^{-1},
\]%

\begin{equation}
g^{\pm}\in C_{0,s}^{\gamma,\gamma/2}(\overline{\Omega}_{T}^{\pm}),
\ f^{\pm}\in
C_{0}^{2+\beta-\alpha,\frac{2+\beta-\alpha}{2-\alpha}}(\Gamma_{T}),
\ h^{\pm}\in
C_{0}^{2+\gamma,\frac{2+\gamma}{2}}(\Gamma_{T}^{\pm}).\label{B.12}%
\end{equation}

\begin{theorem}\label{TB1}

The problem \eqref{B.8}- \eqref{B.11} has the unique solution from
the space $C_{0,s}^{2+\gamma,\frac{2+\gamma}{2}}(\overline{\Omega}
_{T}^{\pm})$ and the following estimate is valid

\begin{equation}
|u^{\pm}|_{s,\overline{\Omega}_{T}^{\pm}}^{(2+\gamma)}\leq C\left(
|g^{\pm}|_{s,\overline{\Omega}_{T}^{\pm}}^{(\gamma)}+|f^{\pm}|_{\Gamma_{T}%
}^{(2+\beta-\alpha,\frac{2+\beta-\alpha}{2-\alpha})}+|h^{\pm}|_{\Gamma_{T}%
^{\pm}}^{(2+\gamma,\frac{2+\gamma}{2})}\right)  .\label{B.12.1}%
\end{equation}

\end{theorem}

\section{The model problem for the two phase Stefan problem for the degenerate
equations.} \label{s3}

Let $a\geq 0$ is a fixed number. Denote $Q^{N}_{\pm}=\{(x,t): x\in
R^{N}_{\pm},t\geq -a\}$, $Q^{N-1}=\{(x',t): x'\in R^{N-1},t\geq
-a\}$. Denote further $R^{N,a}_{\pm,T}=Q^{N}_{\pm}\cap \{t\leq T\}$,
$R^{N-1,a}_{T}=Q^{N-1}\cap \{t\leq T\}$. It is convenient to
consider the domains with the $t \geq -a$, as it will allow us to
consider the points with $t = 0$ as interior points of general
position, which will facilitate the further notation.  We agree,
which is similar to \eqref{B.2}, that zero at the bottom of the
designation of the spaces of functions defined in these domains
means the subspace of the corresponding space whose elements vanish
at $t=-a$ together with all its derivatives with respect to $t$,
which are permitted by the space.

Let $f(x',t)$ is a finite with respect to  $x$ function, which is
defined in $Q^{N-1}$ and is such that

\begin{equation}
f(x^{\prime},-a)\equiv0,\quad f\in
C_{0}^{1+\beta-\alpha,\frac{1+\beta-\alpha
}{2-\alpha}}(\overline{Q}^{N-1}),\label{M.1}%
\end{equation}
which allows us to consider $f$ as the functions, which is defined
for $t\in (-\infty,\infty)$, extending it by identical zero in the
domain $t<-a$ with the preservation of the class.

Let further

\begin{equation}
f_{1}^{\pm}(x,t)\in C_{s,0}^{\gamma,\gamma/2}(\overline{Q}_{\pm}^{N}),\quad f_{2}%
^{\pm}(x^{\prime},t)\in C^{2+\beta-\alpha,1+\gamma/2}_{0}(\overline{Q}%
^{N-1})\label{M1.1}%
\end{equation}
- are given finite functions which are also extended by identical
zero in the domain $t<-a$ with the preservation of the classes.

Consider the following model problem for the triple of the unknown
functions $u^{\pm}(x,t)$ and $\rho(x',t)$, which are defined in
$\overline{Q}^{N}_{\pm}$ and $\overline{Q}^{N-1}$ correspondingly:

\begin{equation}
\frac{\partial u^{\pm}}{\partial t}-(\pm x_{N})^{\alpha}\Delta u^{\pm}=f_{1}%
^{\pm},\quad(x,t)\in Q_{\pm}^{N},\label{M.2}%
\end{equation}%

\begin{equation}
u^{\pm}+A^{\pm}\rho=f_{2}^{\pm},\quad x_{N}=0,(x^{\prime},t)\in Q^{N-1},\label{M.3}%
\end{equation}%

\begin{equation}
\rho_{t}-\varepsilon\Delta_{x^{\prime}}\rho+b^{+}\frac{\partial u^{+}%
}{\partial x_{N}}-b^{-}\frac{\partial u^{-}}{\partial
x_{N}}=f(x^{\prime
},t),\quad x_{N}=0,(x^{\prime},t)\in Q^{N-1},\label{M.4}%
\end{equation}%

\begin{equation}
u^{\pm}(x,-a)=0,\quad\rho(x^{\prime},-a)=0,\label{M.5}%
\end{equation}%

\begin{equation}
u^{\pm}\in C_{s,0}^{2+\gamma,\frac{2+\gamma}{2}}(\overline{Q}_{\pm}^{N}%
),\quad\rho\in C^{2+\beta-\alpha,1+\gamma/2}_{0}(\overline{Q}%
^{N-1}), \rho_{t}\in C_{0}^{1+\beta-\alpha,\frac{1+\beta-\alpha
}{2-\alpha}}(\overline{Q}^{N-1}),\label{M.6}%
\end{equation}
where $A^{\pm}$, $b^{\pm}$, $\varepsilon$ are given positive
constants  and $\varepsilon \in (0,1)$.

Note that the term with $\varepsilon$ in \eqref{M.4} does not apply
directly to the Stefan problem and serves as a regularization of the
problem, that will be needed in the proof of the solvability of the
corresponding linearized Stefan problem in an arbitrary domain. To
the author's knowledge, this regularization of the boundary
condition in the Stefan problem was first used in the paper
\cite{1}.

Below we prove the following a priori estimate of the solution of
the problem \eqref{M.2} - \eqref{M.6}.

\begin{theorem} \label{TM1}

Let $u^{\pm}(x,t)\in
C_{s,0}^{2+\gamma,\frac{2+\gamma}{2}}(\overline{Q} _{\pm}^{N})$,
$\rho\in
C^{3+\beta-\alpha,1+\frac{1+\beta-\alpha}{2-\alpha}}_{0}(\overline{Q}^{N-1})$
 are a finite solution of the problem \eqref{M.2}- \eqref{M.6}. Then for arbitrary $T>0$
 the following estimate is valid:

\[
\mathcal{U}(T)\equiv|u^{+}|_{s,R_{+T}^{N,a}}^{(2+\gamma)}+|u^{-}|_{s,R_{-T}^{N,a}%
}^{(2+\gamma)}+\varepsilon\sum\limits_{i,j=1}%
^{N-1,a}|\rho_{x_{i}x_{j}}|_{C^{1+\beta-\alpha,\frac{1+\beta-\alpha}{2-\alpha}%
}(R_{T}^{N-1,a})}+
\]%
\[
+
|\rho|^{(2+\beta-\alpha,\frac{2+\beta-\alpha}{2-\alpha})}_{R^{N-1,a}_{T}}+
|\rho_{t}|^{(1+\beta-\alpha,\frac{1+\beta-\alpha}{2-\alpha})}_{R^{N-1,a}_{T}}
\leq
\]
\[
\leq C_{T}\left(  |f_{1}^{+}|_{s,R_{+T}^{N,a}}^{(\gamma)}+|f_{1}^{-}%
|_{s,R_{-T}^{N,a}}^{(\gamma)}+|f_{2}^{+}|_{C^{2+\beta-\alpha,\frac
{2+\beta-\alpha}{2-\alpha}}(R_{T}^{N-1,a})}+\right.
\]%

\begin{equation}
\left.  +|f_{2}^{-}|_{C^{2+\beta-\alpha,\frac{2+\beta-\alpha}{2-\alpha}}%
(R_{T}^{N-1,a})}+|f|_{C^{1+\beta-\alpha,\frac{1+\beta-\alpha}{2-\alpha}}%
(R_{T}^{N-1,a})}\right)  \equiv C_{T}\mathcal{M}(T),\label{M.7}%
\end{equation}
where the constant $C_{T}$ in \eqref{M.7} does not depend on
$\varepsilon\in (0,1)$.

\end{theorem}

Subsequent content of this section is the proof of the  theorem
 \ref{TM1}.

Note first that by lemmas \ref{LPA1} and \ref{LSl1} we can without
loss of generality assume that

\begin{equation}
f_{1}^{\pm}\equiv0,\quad f_{2}^{\pm}\equiv0,\label{M.8}%
\end{equation}
since the general case can be reduced to the specified one by the
change of the unknown functions $u^{\pm} = v^{\pm} + w^{\pm}$, where
$v^{\pm}$ are the new unknowns, and $w^{\pm}$ satisfy \eqref{M.2}
with the boundary conditions

\begin{equation}
w^{\pm}|_{x_{N}=0}=f_{2}^{\pm}(x^{\prime},t).\label{M.9}%
\end{equation}

Thus, further we assume that only the function $f(x',t)$ is nonzero
in the righthand sides of \eqref{M.2} - \eqref{M.4}.

In addition, because the right side of the relations \eqref{M.2} -
\eqref{M.4} belong to the classes with zero at the bottom and
because of conditions \eqref{M.5}, \eqref{M.6} we can consider that
the relation \eqref{M.2} - \eqref{M.4} are valid and for $t<-a$,
assuming that all the functions are extended by zero to this domain.

An important point of proving \eqref{M.7} is to prove the following
a priori estimate.

\begin{lemma} \label{LM1}
Under the conditions of the theorem \ref{TM1} and under the
condition \eqref{M.8} the following estimate is valid

\[
\left\langle \nabla_{x^{\prime}}\rho\right\rangle _{x^{\prime},R_{T}^{N-1,a}%
}^{(1+\beta-\alpha)}\leq C_{T}\left(  |\nabla_{x^{\prime}}u^{+}|_{R_{+T}^{N,a}%
}^{(0)}+|\nabla_{x^{\prime}}u^{-}|_{R_{-T}^{N,a}}^{(0)}+|\nabla_{x^{\prime}}%
\rho|_{R_{T}^{N-1,a}}^{(0)}+\mathcal{M}(T)\right)  \equiv
\]%

\begin{equation}
\equiv C_{T}\left(  \mathcal{N}(T)+\mathcal{M}(T)\right)  \leq C_{T}(T+a)^{\delta}\mathcal{U}(T)+C_{T}%
\mathcal{M}(T),\label{M.10}%
\end{equation}
where

\[
\mathcal{N}(T)\equiv|\nabla_{x^{\prime}}u^{+}|_{R_{+T}^{N,a}}^{(0)}+|\nabla_{x^{\prime}%
}u^{-}|_{R_{-T}^{N,a}}^{(0)}+|\nabla_{x^{\prime}}\rho|_{R_{T}^{N-1,a}}^{(0)}.
\]

\end{lemma}

To obtain the last inequality in \eqref{M.10} we use the estimates
\eqref{B.3}- \eqref{B.6}.

\begin{proof}

Denote for brevity, $l=1+\beta-\alpha$ and fix a point $(x'_{0},
t_{0})$ in the set $R^{N-1,a}_T$. In order to maintain the
succession of the notations with the paper \cite{17}, whose method
we're going to apply, without loss of generality, we will assume
that $(x'_{0}, t_{0})=(0,0)$ - this choice is not important, as can
be seen from the following proof. Suppose, further, $O=(x'=0,
x_{N}=0, t=0)$ is the corresponding point in $R^{N, a}_{+T}$. We
show that for every $h\in (0,1) $ and for any $i,j=\overline{1,N-1}$
we have the following inequality

\begin{equation}
|\rho_{x_{i}}(\overrightarrow{e}_{j}h,0)-\rho_{x_{i}}(-\overrightarrow{e}%
_{j}h,0)|\leq C_{T}\left(  \mathcal{N}(T)+\mathcal{M}(T)\right)  h^{l},\label{M.11}%
\end{equation}
where $\overrightarrow{e}_{j}$ is the unit vector of the $Ox_{j}$-
axis.  Since the point $O$ and the step $h\in (0,1)$ in the relation
\eqref{M.11} are arbitrary, the estimate \eqref{M.10} of the lemma
follows from  the estimate \eqref{M.11}.

So, let $y_{1},y_{2}\in [0,1]$, $y\equiv(y_{1},y_{2})$ and let also
$i,j\in \{1,2,...,N-1\}$ are fixed. Consider the differences

\begin{equation}
v^{\pm}(x,t,y_{1},y_{2})=\Delta_{i,y_{1}}\Delta_{j,y_{2}}u^{\pm}(x,t)=\label{M.12}%
\end{equation}%

\[
=u^{\pm}(x+y_{1}\overrightarrow{e}_{i}+y_{2}\overrightarrow{e}_{j}%
,t)-u^{\pm}(x-y_{1}\overrightarrow{e}_{i}+y_{2}\overrightarrow{e}_{j},t)-
\]%

\[
-u^{\pm}(x+y_{1}\overrightarrow{e}_{i}-y_{2}\overrightarrow{e}_{j}%
,t)+u^{\pm}(x+y_{1}\overrightarrow{e}_{i}+y_{2}\overrightarrow{e}_{j},t),
\]
where

\begin{equation}
\Delta_{k,h}u(x,t)\equiv
u(x+h\overrightarrow{e}_{k},t)-u(x-h\overrightarrow
{e}_{k},t).\label{M.12+1}%
\end{equation}
Denote also

\begin{equation}
r(x^{\prime},t,y_{1},y_{2})=\Delta_{i,y_{1}}\Delta_{j,y_{2}}\rho(x^{\prime
},t).\label{M.14}%
\end{equation}

Note that

\begin{equation}
\frac{\partial^{2}v^{\pm}}{\partial x_{i}^{2}}-\frac{\partial^{2}v^{\pm}%
}{\partial y_{1}^{2}}=0,\quad\frac{\partial^{2}v^{\pm}}{\partial x_{j}^{2}%
}-\frac{\partial^{2}v^{\pm}}{\partial y_{2}^{2}}=0.\label{M.15}%
\end{equation}
Therefore in domains $R^{N}_{\pm}\times \{-\infty<t< T\}\times
\{0<y_{1}<1\}\times \{0<y_{2}<1\}$ the functions $v^{\pm}(x,t,y)$
satisfy the equations

\begin{equation}
L^{\ast}v^{\pm}\equiv\frac{\partial v^{\pm}}{\partial t}-\label{M.16}%
\end{equation}%

\[
-(\pm x_{N})^{\alpha}\left(  \sum\limits_{k\neq i,j}\frac{\partial^{2}v^{\pm}%
}{\partial x_{k}^{2}}+\frac{3}{4}\frac{\partial^{2}v^{\pm}}{\partial x_{i}^{2}%
}+\frac{3}{4}\frac{\partial^{2}v^{\pm}}{\partial x_{j}^{2}}+\frac{1}{4}%
\frac{\partial^{2}v^{\pm}}{\partial
y_{1}^{2}}+\frac{1}{4}\frac{\partial ^{2}v^{\pm}}{\partial
y_{2}^{2}}\right)  =0,
\]

Note also, that

\[
|v^{\pm}|=|\Delta_{i,y_{1}}\Delta_{j,y_{2}}u^{\pm}(x,t)|=
\]%

\[
=y_{2}\left|
\Delta_{i,y_{1}}\int\limits_{-1}^{1}u_{x_{j}}^{\pm}(x+\omega
y_{2}\overrightarrow{e}_{j})d\omega\right|
\leq4y_{2}\mathcal{N}(T).
\]
Exactly the same way

\[
|v^{\pm}|\leq4y_{1}\mathcal{N}(T),
\]
and therefore

\begin{equation}
|v^{\pm}|\leq4y_{\min}\mathcal{N}(T),\label{M.16.1}%
\end{equation}
where

\begin{equation}
y_{\min}=\min\left\{  y_{1},y_{2}\right\}  .\label{M.16.2}%
\end{equation}

Similarly, we have

\begin{equation}
|r|\leq4y_{\min}\mathcal{N}(T),\label{M.16.1.1}%
\end{equation}

Denote

\begin{equation}
y=(y_{1},y_{2}),\quad P^{\pm}=\left\{
(x,t,y):|x_{m}|<1,m=\overline
{1,N-1},0<\pm x_{N}<1,\right. \label{M.17}%
\end{equation}%

\[
\left.  -1<t<0,0<y_{k}<1,k=1,2\right\}  .
\]
Denote also

\begin{equation}
\Sigma^{\pm}=\partial P^{\pm}\setminus\left(  \left\{  t=0\right\}
\cup\left\{ x_{N}=0\right\}  \right)  ,\quad\Sigma_{0}=\partial
P^{\pm}\cap\left\{
x_{N}=0\right\} ,\label{M.17.1}%
\end{equation}
that is $\Sigma^{\pm}$ - are parabolic boundaries of the
parallelepipeds  $P^{\pm}$ without their common part $\{x_{N}=0\}$,
and the last will be denoted by $\Sigma_{0}$.

In the parallelepipeds $P^{\pm}$ the functions $v^{\pm}$ and $r$
satisfy the following problem

\begin{equation}
L^{\ast}v^{\pm}=0,\quad(x,t,y)\in P^{\pm},\label{M.18}%
\end{equation}%

\begin{equation}
v^{\pm}|_{\Sigma^{\pm}}=g^{\pm}(x,t,y),\label{M.19}%
\end{equation}%

\begin{equation}
v^{\pm}+A^{\pm}r=0,\quad x_{N}=0,\label{M.20}%
\end{equation}%

\begin{equation}
r_{t}-\varepsilon\Delta_{x^{\prime}}r+b^{+}\frac{\partial
v^{+}}{\partial x_{N}}-b^{-}\frac{\partial v^{-}}{\partial
x_{N}}=F(x,y,t),\quad
x_{N}=0,\label{M.21}%
\end{equation}
where

\begin{equation}
F(x,t,y)=\Delta_{i,y_{1}}\Delta_{j,y_{2}}f,\quad
g^{\pm}(x,t,y)=\Delta
_{i,y_{1}}\Delta_{j,y_{2}}u^{\pm},\label{M.22}%
\end{equation}
and, in view of the assumptions \eqref{M.1},

\[
|F(x,t,y)|\leq2\left\langle f\right\rangle
_{x_{j}}^{(l)}(2y_{2})^{l}.
\]
A similar inequality with replacing $j$ by $i$ and $y_{2}$ by
$y_{1}$ gives similar to \eqref{M.16.1}

\begin{equation}
|F(x,t,y)|\leq C\mathcal{M}(T)y_{\min}^{l}.\label{M.23}%
\end{equation}
Note also that by \eqref{M.16.1},

\begin{equation}
|g^{\pm}(x,t,y)|\leq C\mathcal{N}(T)y_{\min}.\label{M.24}%
\end{equation}

To estimate $v^{\pm}$ and $r$, we're going to apply to the problem
\eqref{M.18} - \eqref{M.21} the maximum principle in the following
form.

\begin{lemma} \label{LM2}

Let functions $H^{\pm}(x,t,y)\in C^{2,1}(P^{\pm})\cap
C^{1,0}(\overline{P}^{\pm})$, $S(x^{\prime},t,y)\in
C^{2,1}(\Sigma_{0})$ satisfy the conditions

\begin{equation}
L^{\ast}H^{\pm}\geq0,\quad(x,t,y)\in P^{\pm},\label{M.25}%
\end{equation}%

\begin{equation}
H^{\pm}|_{\Sigma^{\pm}}\geq0,\label{M.26}%
\end{equation}%

\begin{equation}
H^{\pm}+A^{\pm}S=0,\quad x_{N}=0,\label{M.27}%
\end{equation}%

\begin{equation}
S_{t}-\varepsilon\Delta_{x^{\prime}}S+b^{+}\frac{\partial
H^{+}}{\partial
x_{N}}-b^{-}\frac{\partial H^{-}}{\partial x_{N}}\leq0,\quad x_{N}%
=0.\label{M.28}%
\end{equation}

Then

\begin{equation}
H^{\pm}\geq0,\quad(x,t,y)\in\overline{P}^{\pm};\quad
S\leq0,\quad(x^{\prime
},t,y)\in\Sigma_{0}.\label{M.29}%
\end{equation}

\end{lemma}

We do not give a detailed proof of this lemma, since it uses
standard arguments. We only note that the functions $H^{\pm}$ can
not reach a negative minimum at $\{x_{N}=0\}$, as in this case, by
\eqref{M.27}, they would reached a negative minimum simultaneously
and corresponding point would be, again by \eqref{M.27}, a point of
a positive maximum of the function $S$. All this together in the
standard way contradicts the boundary condition \eqref{M.28}.

We shall need the the auxiliary functions $w^{\pm}(x,t)$, defined on

\[
\Pi^{\pm}=\left\{  |x_{m}|\leq1,m=\overline{1,N-1},0\leq \pm
x_{N}\leq1,-1\leq t\leq0\right\}
\]
correspondingly, and such that

\[
L_{xt}^{\ast}w^{\pm}=\frac{\partial w^{\pm}}{\partial t}-
\]%

\begin{equation}
-(\pm x_{N})^{\alpha}\left(  \sum\limits_{k\neq i,j}\frac{\partial^{2}w^{\pm}%
}{\partial x_{k}^{2}}+\frac{3}{4}\frac{\partial^{2}w^{\pm}}{\partial x_{i}^{2}%
}+\frac{3}{4}\frac{\partial^{2}w^{\pm}}{\partial x_{j}^{2}}\right)
=0,\quad
x_{N}\neq0,\label{M.30}%
\end{equation}%

\begin{equation}
w^{\pm}|_{\left\{  |x_{k}|=1\right\}  \cup\left\{  t=-1\right\}
}\geq
\nu>0,\label{M.31}%
\end{equation}%

\begin{equation}
w^{\pm}(0,0)=0,\quad w^{\pm}|_{\Pi^{\pm}}\geq0,\label{M.32}%
\end{equation}%

\begin{equation}
w^{\pm}(x,t)\in C^{2,1}(\Pi^{\pm}\cap\left\{  x_{N}=0\right\}  ).\label{M.33}%
\end{equation}

Such functions can be constructed as follows. Consider for example,
$w^{+}(x,t)$. Let $G^{+}(x,t)$ is a  function from $C^{\infty}$ in
$\overline{R^{N}_{+}} \times (-\infty, \infty)$, such that $G^{+}
\equiv 0$ for $|x|+|t| \leq 1/4 $ and for
  $t \leq -2$, $|x| \geq 2$ and $G^{+}>0$ in the other
points of $\overline{R^{N}_{+}} \times (- \infty, \infty)$. Let
$w^{+}(x,t)$ is the solution of the following initial boundary value
problem in half-space

\[
L_{x,t}^{\ast}w^{+}=0,\quad x_{N}>0, \ t>-2,
\]%

\[
w^{+}|_{x_{N}=0}=G^{+}(x,t)\in C^{\infty},
\]%

\[
w^{+}|_{t=-2}=0.
\]

Lemma \ref{LPA1} implies that the function $w^{+}$ exists in the
appropriate class, and

\begin{equation}
|w^{+}|_{s,\overline{R_{+}^{N}}\times\lbrack-2,0]}^{(2+\gamma)}\leq
C(G^{+}).\label{M.34}%
\end{equation}

Because of the properties of $G^{+}(x,t) $ and by the strong maximum
principle (see \cite{22}), the function $w^{+}$ has all desirable
properties, including \eqref{M.31}.


Now consider the following comparison functions defined on
$\overline{P}^{\pm}$. Denote

\begin{equation}
\varphi(y)=y_{1}y_{2}\left(  y_{1}^{l}+y_{2}^{l}\right)  ^{-\frac{1}{l}%
+1},\label{M.37}%
\end{equation}%

\begin{equation}
\psi^{\pm}(x_{N},y)=y_{1}y_{2}\left[  (y_{1}\pm x_{N})^{l}+(y_{2}\pm x_{N}%
)^{l}\right]  ^{-\frac{1}{l}+1},\quad \pm x_{N}\geq0,\label{M.38}%
\end{equation}%

\begin{equation}
\theta^{\pm}(x,t,y)=%
\genfrac{\{}{.}{0pt}{}{\left(  y_{1}^{-1}+y_{2}^{-1}\right)  ^{-1}%
w^{\pm}(x,t),\quad y_{\min}>0,}{0,\quad y_{\min}=0.}%
\label{M.39}%
\end{equation}

The direct verification shows (cf. \cite{17}), that the functions
$\varphi$ and $\psi^{\pm}$ possess properties

\begin{equation}
\pm\frac{\partial\psi^{\pm}}{\partial x_{N}}|_{x_{N}=0}\leq-\nu y_{\min}%
^{l},\label{M.40}%
\end{equation}%

\begin{equation}
|L^{\ast}\psi^{\pm}|\leq C|x_{N}|^{\alpha}y_{\min}^{-1+l},\label{M.41}%
\end{equation}%

\begin{equation}
\frac{\partial\varphi}{\partial x_{N}}|_{x_{N}=0}=0,\label{M.42}%
\end{equation}%

\begin{equation}
L^{\ast}\varphi\geq\nu|x_{N}|^{\alpha}y_{\min}^{-1+l},\label{M.43}%
\end{equation}%

\begin{equation}
\varphi|_{y_{k}=1}\geq\nu y_{\min}.\label{M.43.1}%
\end{equation}

Thus, if we choose a sufficiently large constant $K>0 $, the
functions

\begin{equation}
h^{\pm}\equiv\psi^{\pm}+K\varphi\label{M.44}%
\end{equation}
will have the properties

\begin{equation}
\pm\frac{\partial h^{\pm}}{\partial x_{N}}|_{x_{N}=0}\leq-\nu y_{\min}%
^{l},\label{M.45}%
\end{equation}%

\begin{equation}
L^{\ast}h^{\pm}\geq\nu|x_{N}|^{\alpha}y_{\min}^{-1+l}>0,\quad(x,t,y)\in
P^{\pm},\label{M.46}%
\end{equation}%

\begin{equation}
h^{\pm}|_{y_{k}=1}\geq\nu y_{\min}.\label{M.46.1}%
\end{equation}

At the same time, the functions $\theta^{\pm}(x,t,y)$ have the
properties

\begin{equation}
L^{\ast}\theta^{\pm}\geq0,\quad(x,t,y)\in P^{\pm},\label{M.47}%
\end{equation}%

\begin{equation}
\left|  \frac{\partial\theta^{\pm}}{\partial x_{N}}\right|
_{x_{N}=0}\leq
Cy_{\min},\label{M.48}%
\end{equation}%

\begin{equation}
\theta^{\pm}|_{\cup_{k}\left\{  |x_{k}|=1\right\}  \cup\left\{
t=-1\right\}
}\geq\nu y_{\min}.\label{M.49}%
\end{equation}

Consider now the following comparison functions

\begin{equation}
H^{\pm}(x,t,y)\equiv A^{\pm}\left[  L_{1}\theta^{\pm}(x,t,y)+L_{2}h^{\pm}%
(x_{N},y)\right]  (N(T)+M(T)),\label{M.50}%
\end{equation}%

\begin{equation}
S(x^{\prime},t,y)=-\left[  L_{1}\theta^{\pm}(x^{\prime},0,t,y)+L_{2}%
h^{\pm}(0,y)\right]  (N(T)+M(T)),\label{M.51}%
\end{equation}
where $L_{1}$ и $L_{2}$ are some positive constants.

Choosing first $L_{1}$ and then $L_{2}$ are sufficiently large, and
using on one hand \eqref{M.23}, \eqref{M.24}, and on the other hand
 \eqref{M.45} - \eqref{M.49}, we see that the triple of the
functions

\[
H^{+}\pm v^{+}(x,t,y),\quad H^{-}\pm v^{-}(x,t,y),\quad S \pm
r(x^{\prime},t,y)
\]
satisfies in $P^{\pm}$ to the conditions of the lemma \ref{LM2}.
Hence,

\[
H^{+}\pm v^{+}\geq0,\quad H^{-}\pm v^{-}\geq0,\quad S\pm r\leq0,
\]
that is

\begin{equation}
|v^{\pm}|\leq CH^{\pm},\quad|r|\leq C|S|,\quad(x,t,y)\in\overline{P}%
^{\pm}.\label{M.52}%
\end{equation}

Taking in \eqref{M.52} $x=0$, $t=0$, in view of
$\theta^{\pm}(0,0,y)=0$, we obtain, for example, for $v^{+}$
similarly \cite{17}

\[
|\Delta_{i,y_{1}}\Delta_{j,y_{2}}u^{+}(0,0)|\leq C\left[
\mathcal{N}(T)+\mathcal{M}(T)\right] y_{1}y_{2}\left(
y_{1}^{l}+y_{2}^{l}\right) ^{-\frac{1}{l}+1}.
\]
Dividing both sides of this relation by $y_{1}$ and taking the limit
with $y_{1} \rightarrow 0$, we obtain

\begin{equation}
\left|  \frac{\partial u^{+}}{\partial x_{i}}(0+y_{2}\overrightarrow{e}%
_{j},0)-\frac{\partial u^{+}}{\partial x_{i}}(0-y_{2}\overrightarrow{e}%
_{j},0)\right|  \leq C\left[  \mathcal{N}(T)+\mathcal{M}(T)\right]  y_{2}^{l},\label{M.53}%
\end{equation}
and similarly

\begin{equation}
\left|  \frac{\partial\rho}{\partial x_{i}}(0+y_{2}\overrightarrow{e}%
_{j},0)-\frac{\partial\rho}{\partial x_{i}}(0-y_{2}\overrightarrow{e}%
_{j},0)\right|  \leq C\left[  \mathcal{N}(T)+\mathcal{M}(T)\right]  y_{2}^{l}.\label{M.54}%
\end{equation}

Since all of the above arguments are valid, as noted, for any
$(x'_{0},t_{0}) \in R^{N-1, a}_{T}$, by the same token the estimate
\eqref{M.10} and the lemma \ref{LM1} are proved.

\end{proof}

We continue the proof of the theorem. It follows from \eqref{M.4},
that

\begin{equation}
\rho_{t}-\varepsilon\Delta_{x^{\prime}}\rho=F(x^{\prime},t)\equiv-b^{+}%
\frac{\partial u^{+}}{\partial x_{N}}+b^{-}\frac{\partial
u^{-}}{\partial
x_{N}}+f,\label{M.55}%
\end{equation}
Moreover, in view of the inequalities \eqref{B.7}, \eqref{B.3},
\eqref{B.6}

\[
|F|_{R_{T}^{N-1}}^{(\gamma)}\leq C\left(  |\nabla u^{+}|_{R_{T}^{N-1,a}%
}^{(\gamma)}+|\nabla u^{-}|_{R_{T}^{N-1,a}}^{(\gamma)}\right)
+C\mathcal{M}(T)\leq
\]%

\begin{equation}
\leq C(T+a)^{\mu}\mathcal{U}(T)+C\mathcal{M}(T),\label{M.56}%
\end{equation}
and

\[
F(x^{\prime},0)\equiv0.
\]

Making in the problem \eqref{M.55}, \eqref{M.5} the change of
variables $x^{\prime} = \varepsilon^{1/2}y$, we obtain the problem

\begin{equation}
\rho_{t}-\Delta_{y}\rho=\widetilde{F}(y,t),\quad(y,t)\in
R_{T}^{N-1,a},\quad
\rho(y,-a)=0,\label{M.57}%
\end{equation}
and

\begin{equation}
\left\langle \widetilde{F}(y,t)\right\rangle
_{t,R_{T}^{N-1,a}}^{(\gamma /2)}=\left\langle
F(x^{\prime},t)\right\rangle _{t,R_{T}^{N-1,a}}^{(\gamma
/2)},\left\langle \widetilde{F}(y,t)\right\rangle
_{y,R_{T}^{N-1,a}}^{(\gamma )}=\varepsilon^{\gamma/2}\left\langle
F(x^{\prime},t)\right\rangle
_{x^{\prime},R_{T}^{N-1,a}}^{(\gamma)}.\label{M.58}%
\end{equation}

It follows from the arguments of \cite{12}, гл.IV, that

\begin{equation}
\left\langle \rho_{y_{i}y_{j}}\right\rangle
_{y,R_{T}^{N-1,a}}^{(\gamma)}\leq C\left\langle
\widetilde{F}(y,t)\right\rangle _{y,R_{T}^{N-1,a}}^{(\gamma
)},\label{M.59}%
\end{equation}%

\begin{equation}
\left\langle \rho_{t}\right\rangle
_{t,R_{T}^{N-1,a}}^{(\gamma/2)}\leq C\left( \left\langle
\widetilde{F}(y,t)\right\rangle _{y,R_{T}^{N-1,a}}^{(\gamma
)}+\left\langle \widetilde{F}(y,t)\right\rangle
_{t,R_{T}^{N-1,a}}^{(\gamma
/2)}\right)  .\label{M.60}%
\end{equation}
Making in \eqref{M.59}, \eqref{M.60} the inverse change of
variables, in view of \eqref{M.58} we obtain

\begin{equation}
\left\langle \rho_{t}(x^{\prime},t)\right\rangle
_{t,R_{T}^{N-1,a}}^{(\gamma
/2)}+\varepsilon\sum\limits_{i,j}\left\langle
\rho_{x_{i}x_{j}}\right\rangle
_{x^{\prime},R_{T}^{N-1,a}}^{(\gamma)}\leq C|F(x^{\prime},t)|_{R_{T}^{N-1,a}%
}^{(\gamma)}.\label{M.61}%
\end{equation}

Thus, in view of the estimate \eqref{M.10} of the lemma  \ref{LM1},
it is proved, that

\[
|\rho_{t}|_{R_{T}^{N-1,a}}^{(\gamma)}+\left\langle \nabla_{x^{\prime}}%
\rho\right\rangle
_{x^{\prime},R_{T}^{N-1,a}}^{(1+\beta-\alpha)}+\varepsilon
|\rho|_{R_{T}^{N-1,a}}^{(2+\gamma)}\leq
C(T+a)^{\mu}\mathcal{U}(T)+C\mathcal{M}(T),
\]
or, in view of \eqref{Sol.1}, \cite{15} and of the finiteness of
$\rho$,

\begin{equation}
|\rho|_{C^{2+\beta-\alpha,\frac{2+\beta-\alpha}{2-\alpha}}(R_{T}^{N-1,a}%
)}+\varepsilon|\rho|_{R_{T}^{N-1,a}}^{(2+\gamma)}\leq C(T+a)^{\mu}%
\mathcal{U}(T)+C\mathcal{M}(T),\label{M.62}%
\end{equation}
where the constant $C$ does not depend on $\varepsilon>0$.

Now, considering $u^{\pm}(x,t)$ as the solution of the
Cauchy-Dirichlet problem \eqref{M.2}, \eqref{M.3}, \eqref{M.5}, by
the lemma \ref{LPA1} and the estimate \eqref{M.62}, we conclude that

\begin{equation}
|u^{+}|_{s,R_{+T}^{N,a}}^{(2+\gamma)}+|u^{-}|_{s,R_{-T}^{N,a}}^{(2+\gamma)}\leq
C(T+a)^{\mu}\mathcal{U}(T)+C\mathcal{M}(T).\label{M.63}%
\end{equation}

It follows that in the condition \eqref{M.4}

\[
\left|  \frac{\partial u^{+}}{\partial x_{N}}\right|  _{C^{1+\beta
-\alpha,\frac{1+\beta-\alpha}{2-\alpha}}(R_{T}^{N-1,a})}+\left|
\frac{\partial u^{-}}{\partial x_{N}}\right|
_{C^{1+\beta-\alpha,\frac{1+\beta-\alpha
}{2-\alpha}}(R_{T}^{N-1,a})}\leq
C(T+a)^{\mu}\mathcal{U}(T)+C\mathcal{M}(T).
\]

Thus, the function $\rho(x',t)$ satisfies the Cauchy problem
\eqref{M.55}, \eqref{M.5} with the right hand side $F$ and the last
has the property

\begin{equation}
F(x^{\prime},-a)=0,\quad\left|  F\right|  _{C^{1+\beta-\alpha,\frac
{1+\beta-\alpha}{2-\alpha}}(R_{T}^{N-1,a})}\leq
C(T+a)^{\mu}\mathcal{U}(T)+C\mathcal{M}(T).\label{M.64}
\end{equation}

Making again in \eqref{M.55}, \eqref {M.5} the change of variables
$x^{\prime}=\varepsilon^{1/2}y$, we arrive at the problem of the
form \eqref{M.57} with $\widetilde {F} $, where the last is such
that

\begin{equation}
\left\langle \widetilde{F}\right\rangle _{t,R_{T}^{N-1,a}}^{(\frac
{1+\beta-\alpha}{2-\alpha})}=\left\langle F\right\rangle _{t,R_{T}^{N-1,a}%
}^{(\frac{1+\beta-\alpha}{2-\alpha})},\quad\left\langle \widetilde
{F}\right\rangle _{y,R_{T}^{N-1,a}}^{(1+\beta-\alpha)}=\varepsilon
^{\frac{1+\beta-\alpha}{2}}\left\langle F\right\rangle _{x^{\prime}%
,R_{T}^{N-1,a}}^{(1+\beta-\alpha)}.\label{M.65}%
\end{equation}
As above, completely similar to \cite{12}, Ch.IV, for solutions of
the problem \eqref{M.57} we have the estimates

\begin{equation}
\left\langle \rho_{t}\right\rangle
_{t,R_{T}^{N-1,a}}^{(\frac{1+\beta-\alpha
}{2-\alpha})}\leq C\left\langle \widetilde{F}\right\rangle _{t,R_{T}^{N-1,a}%
}^{(\frac{1+\beta-\alpha}{2-\alpha})}\leq C(T+a)^{\mu}\mathcal{U}(T)+C\mathcal{M}(T),\label{M.66}%
\end{equation}%

\begin{equation}
\sum\limits_{i,j=1}^{N-1}\left\langle
\rho_{y_{i}y_{j}}\right\rangle
_{y,R_{T}^{N-1,a}}^{(1+\beta-\alpha)}\leq C\left\langle \widetilde
{F}\right\rangle _{y,R_{T}^{N-1,a}}^{(1+\beta-\alpha)},\label{M.67}%
\end{equation}
and we note that in obtaining the estimate \eqref{M.66} the
condition $F(x',-a)=0$ is important.

Proceeding as before and going back to the variables $x'$, we find
from \eqref{M.67} and \eqref {M.65} that

\begin{equation}
\varepsilon\sum\limits_{i,j=1}^{N-1}\left\langle \rho_{x_{i}x_{j}%
}\right\rangle _{x^{\prime},R_{T}^{N-1,a}}^{(1+\beta-\alpha)}\leq
C\left\langle F\right\rangle
_{x^{\prime},R_{T}^{N-1,a}}^{(1+\beta-\alpha)}\leq C(T+a)^{\mu
}\mathcal{U}(T)+C\mathcal{M}(T).\label{M.68}%
\end{equation}

Now combining the estimates \eqref{M.68}, \eqref{M.66}, \eqref{M.63}
and \eqref{M.62}, we find that

\begin{equation}
\mathcal{U}(T)\leq C(T+a)^{\mu}\mathcal{U}(T)+C\mathcal{M}(T).\label{M.69}%
\end{equation}

Taking now in \eqref{M.69} $T=T_{0}$, so that the value of
$T_{0}+a>0$ is sufficiently small, we obtain estimate \eqref{M.7} on
the interval $[-a, T_{0}]$. Considering further the problem
\eqref{M.2} - \eqref{M.6} on the interval
$[-a+(a+T_{0})/2,-a+3(a+T_{0})/2]$ and removing the initial data
with the known functions, that is moving along the axis of $Ot$ up,
exactly as in \cite{12}, Ch.IV, we obtain the assertion of the
theorem \ref{TM1} on an arbitrary time interval $[-a,T]$.

Thus, the theorem \ref{TM1} is proved. $\square$


\section{Reduction of the problem \eqref{1.3}-\eqref{1.7} to the problem in the fixed
domain.} \label{s4}

Let $\rho(\omega,t)$ is the unknown function defined in Section
\ref{s1} and parameterizing unknown (free) boundary
$\Gamma_{\rho,T}$ ($\rho(\omega,0) \equiv 0$), and let
$\rho(x,t)=E\rho(\omega,t)$ is the extension of this function to the
whole domain $\overline{\Omega_{T}}$ by the extension operator $E$
from \eqref{P.510}.

We pass in the problem \eqref{1.3}- \eqref{1.7} from the unknown
functions $u^{\pm}(y,\tau)$ to the unknowns
$v^{\pm}(y,\tau)=|u^{\pm}|^{m-1}u^{\pm}(y,\tau)$. Then the relations
\eqref{1.3}-\eqref{1.7} take the form:

\begin{equation}
L_{0}(v^{\pm})v^{\pm}\equiv \frac{\partial v^{\pm}}{\partial
\tau}-a^{\pm}|v^{\pm}|^{\alpha}\nabla_{y}^{2}v^{\pm}(y,\tau)=0, \
(y,\tau)\in \Omega_{\rho,T}^{\pm},
 \label{4.3}
\end{equation}
\begin{equation}
v^{+}(y,\tau)=v^{-}(y,\tau)=0, \ \ (y,\tau)\in \Gamma_{\rho,T},
 \label{4.4}
\end{equation}
\begin{equation}
a^{+}\sum_{i=1}^{N}\cos(\overrightarrow{N},y_{i})v^{+}_{y_{i}}-
a^{-}\sum_{i=1}^{N}\cos(\overrightarrow{N},y_{i})v^{-}_{y_{i}}=k\cos(\overrightarrow{N},\tau)
, \ \ (y,\tau)\in \Gamma_{\rho,T},
 \label{4.5}
\end{equation}
\begin{equation}
v^{\pm}(y,\tau)=|g^{\pm}|^{m-1}g^{\pm}(y,\tau)\equiv
h^{\pm}(y,\tau), \ \ (y,\tau)\in \Gamma_{T}^{\pm},
 \label{4.6}
\end{equation}
\begin{equation}
v^{\pm}(y,0)=|u^{\pm}_{0}|^{m-1}u^{\pm}_{0}(y)\equiv
v^{\pm}_{0}(y), \ \ y \in \overline{\Omega^{\pm}}.
 \label{4.7}
\end{equation}

We make in the problem \eqref{4.3}-\eqref{4.7} the change of
variables $(y,\tau)=e_{\rho}(x,t)$ which is defined in \eqref{4.1}.
Denote for simplicity by the same symbols $v^{\pm}(x, t)$ the
unknown functions after this change of variables, that is,

\[
v^{\pm}(x,t)\equiv v^{\pm}(y,\tau)\circ e_{\rho}(x,t).
\]
Then, in view of the properties of $(y,\tau)=e_{\rho}(x,t)$, in the
variables $(x,t)$ the problem \eqref{4.3}-\eqref{4.7} reduces to the
following problem in the known fixed domains $\Omega_{T}^{\pm}$ for
the unknown functions $v^{+}$, $v^{-}$, $\rho$ (besides
$x$-variables  we use the corresponding coordinates
$(\omega,\lambda)$, which were introduced in \eqref{ab1.1}):

\begin{equation}
L_{\rho}(v^{\pm})v^{\pm}\equiv \frac{\partial v^{\pm}}{\partial
t}-h^{\pm}_{\rho}\rho_{t}
 -|v^{\pm}|^{\alpha}\nabla_{\rho}^{2}v^{\pm}(x,t)=0, \ (x,t)\in
\Omega_{T}^{\pm},
 \label{4.9}
\end{equation}
\begin{equation}
v^{+}(x,t)=v^{-}(x,t)=0, \ \ (x,t)\in \Gamma_{T},
 \label{4.10}
\end{equation}
\begin{equation}
(1+\sum_{i,j=1}^{N-1}m_{ij}(x,\rho)\rho_{\omega_{i}}\rho_{\omega_{j}})(a^{+}\frac{\partial
v^{+}}{\partial \lambda}-a^{-}\frac{\partial v^{-}}{\partial
\lambda})=-k\rho_{t}(1+\rho_{\lambda}), \ \ (x,t)\in \Gamma_{T},
 \label{4.11}
\end{equation}
\begin{equation}
v^{\pm}(x,t)= h^{\pm}(x,t), \ \ (x,t)\in \Gamma_{T}^{\pm},
 \label{4.12}
\end{equation}
\begin{equation}
v^{\pm}(x,0)= v^{\pm}_{0}(x), \ x \in \overline{\Omega^{\pm}}, \ \
\rho(\omega,0)\equiv 0,
 \label{4.13}
\end{equation}
\begin{equation}
\rho(x,t)=E\rho(\omega,t),
 \label{a4.1}
\end{equation}
where $\nabla_{\rho}\equiv \mathcal{E}_{\rho}\nabla_{x}$,
and the matrix$\mathcal{E}_{\rho}$ is the conjugate and inverse to Jacobi matrix
of the mapping \eqref{4.1} for $t=const$, $m_{ij}(x,\rho)$ are some given smooth functions
of their arguments, and

\begin{equation}
h^{\pm}_{\rho}(x,t)\equiv \frac{\partial v^{\pm}}{\partial
\lambda}\frac{1}{1+\rho_{\lambda}}.
 \label{4.14}
\end{equation}
Note that the last definition is legitimate, since the function $\rho(x,t)$ is not identically
zero only if $x\in\mathcal{N}$, where the coordinates $(\omega,\lambda) \equiv (\omega_{x},\lambda_{x})$
of the point $x$ are defined,  and the coordinate $\lambda$ is independent of the choice of local
coordinates $\omega $ (we use the index $(\omega_{x}, \lambda_{x})$ to distinguish these
coordinates for a point $x$ from the corresponding coordinates $(\omega_{y},\lambda_{y})$ for a point $y$).

Below we explain the derivation of the relations \eqref{4.9}-\eqref{4.13}, here we note the following.
The relation \eqref{4.11} contains the expression

\[
S_{\rho}\equiv S_{\rho}(\omega,\rho,\rho_{\omega})\equiv (1+\sum_{i,j=1}^{N-1}m_{ij}(x,\rho)\rho_{\omega_{i}}\rho_{\omega_{j}}),
\]
which is explicitly expressed in the local coordinates $\omega$.
But, in fact, the expression $S_{\rho}$ is strictly a function of the points of the surface $\Gamma_{T}$ and
its values at the points of $\Gamma_{T}$ does not depend on a choice of local coordinates $\omega $.
Indeed, first, for any choice of local coordinates $\omega$
the condition \eqref{4.11} is equivalent to \eqref{4.5}, which is independent of a choice of local coordinates,
and, secondly, all the other factors and the terms but $S_{\rho}$ in the relation \eqref{4.11} are invariant
with respect to a choice of $\omega$ and they are the function of the point of the surface $\Gamma_{T}$ only.
Hence, the expression $S_{\rho}$, as a function of the point of the surface $\Gamma_{T}$, is invariant on
a choice of local coordinates $\omega$ as well. And thus, the map $\rho \rightarrow S_{\rho}(\omega,\rho,\rho_{\omega})$ defines a nonlinear operator, acting on functions defined on $\Gamma_{T}$. This operator is invariant under choice of
local coordinates $\omega $, it acts in the space of functions on $\Gamma_{T}$ and has a certain expression $S_{\rho}(\omega,\rho,\rho_{\omega})$ for every particular choice of the local coordinates $\omega$.

Further, the expression $\frac{\partial v^{\pm}}{\partial t}-h^{\pm}_{\rho}\rho_{t}$ is the recalculated in the variables $(x,t)$ derivative $\frac{\partial v^{\pm}}{\partial \tau}$ after the change of variables \eqref{4.1}:

\[
\frac{\partial v^{\pm}}{\partial \tau}=\frac{\partial
v^{\pm}}{\partial t}\frac{\partial t}{\partial
\tau}+\sum_{i=1}^{N-1}\frac{\partial v^{\pm}}{\partial
\omega_{xi}}\frac{\partial \omega_{xi}}{\partial
\tau}+\frac{\partial v^{\pm}}{\partial \lambda_{x}}\frac{\partial
\lambda_{x}}{\partial \tau}.
\]
Here in fact

\begin{equation}
\frac{\partial t}{\partial \tau}=1, \ \ \frac{\partial
\omega_{xi}}{\partial \tau}=0,
 \label{4.15}
\end{equation}
and for the value of $\frac{\partial \lambda_{x}}{\partial \tau}$, due to the relation

\[
\lambda_{x}=\lambda_{y}-\rho(x,t)\circ e_{\rho}^{-1},
\]
and taking into account \eqref{4.15}, we have

\[
\frac{\partial \lambda_{x}}{\partial
\tau}=-\frac{\partial}{\partial \tau}[\rho(x,t)\circ
e_{\rho}(x,t)^{-1}]=
\]
\[
-\frac{\partial \rho}{\partial t}\frac{\partial t}{\partial
\tau}-\frac{\partial \rho}{\partial \lambda_{x}}\frac{\partial
\lambda_{x}}{\partial \tau}-\sum_{i=1}^{N-1}\frac{\partial
\rho}{\partial \omega_{xi}}\frac{\partial \omega_{xi}}{\partial
\tau}= -\rho_{t}-\rho_{\lambda_{x}}\frac{\partial
\lambda_{x}}{\partial \tau}.
\]
So in the variables $x$ and $t$

\begin{equation}
\frac{\partial \lambda_{x}}{\partial
\tau}=-\rho_{t}/(1+\rho_{\lambda_{x}}).
 \label{4.16}
\end{equation}
Thus, it follows from \eqref{4.15} and \eqref{4.16}, that

\[
\frac{\partial v^{\pm}}{\partial \tau}\circ
e_{\rho}=\frac{\partial v^{\pm}}{\partial t}-[\frac{\partial
v^{\pm}}{\partial
\lambda}/(1+\rho_{\lambda})]\rho_{t}=\frac{\partial
v^{\pm}}{\partial t}-h^{\pm}_{\rho}\rho_{t}.
\]

We explain further the transition from the condition \eqref{4.5} to
the condition \eqref{4.11} under the change of variables
\eqref{4.1}, as we shall need in the future the exact explicit form
of this condition. Define in the neighborhood $\mathcal{N}_{T}$ of
the surface $\Gamma_{T}$ the function

\begin{equation}
\Phi_{\rho}(y,\tau)=\lambda_{x}\circ
e_{\rho}^{-1}(y,\tau)=\lambda_{y}-\rho(x,t)\circ
e_{\rho}^{-1}(y,\tau)=\lambda(y)-\rho(y,\tau),
 \label{4.17}
\end{equation}
where for simplicity we have retained for the function $\rho(x,t)
\circ e_{\rho}^{-1}(y,\tau)$ the same notation $\rho(y,\tau)$. By
the definition $\pm\Phi_{\rho}(y,\tau)>0$ for $(y,\tau)\in
\Omega^{\pm}_{\rho,T}$ and $\Phi_{\rho}(y,\tau)=0$ for $(y,\tau)\in
\Gamma_{\rho,T}$. Hence in \eqref{4.5}

\[
\cos(\overrightarrow{N},y_{i})=\frac{\Phi_{\rho
y_{i}}}{|\nabla_{(y,\tau)}\Phi_{\rho}|}, \ \
\cos(\overrightarrow{N},\tau)=\frac{\Phi_{\rho
\tau}}{|\nabla_{(y,\tau)}\Phi_{\rho}|}.
\]
Therefore, the relation \eqref{4.5} can be written as follows

\begin{equation}
a^{+}(\nabla_{y}v^{+},\nabla_{y}\Phi_{\rho})-
a^{-}(\nabla_{y}v^{-},\nabla_{y}\Phi_{\rho})=k\Phi_{\rho \tau}.
 \label{4.18}
\end{equation}
Under the change of variables \eqref{4.1} the right hand side of
\eqref{4.18}, due to the definition of $\Phi_{\rho}$, takes the form

\begin{equation}
k\Phi_{\rho \tau} = k\frac{\partial \lambda_{x}}{\partial
\tau}=-k\rho_{t}/(1+\rho_{\lambda_{x}}),
 \label{4.19}
\end{equation}
owing to \eqref{4.16}.

On the other hand, under the change of variables \eqref{4.1}

\begin{equation}
(\nabla_{y}v^{\pm},\nabla_{y}\Phi_{\rho})\circ
e_{\rho}(x,t)=(\nabla_{\rho}v^{\pm},\nabla_{\rho}\lambda_{x}).
 \label{4.20}
\end{equation}
Denote by $\Lambda(x)$ the transition matrix from the gradient with
respect to the variables $x$ to the gradient with respect to
variables $(\omega_{x},\lambda_{x})$, that is

\begin{equation}
\nabla_{x}=\Lambda(x)\nabla_{(\lambda_{x},\omega_{x})} \ \
(\nabla_{y}=\Lambda(y)\nabla_{(\lambda_{y},\omega_{y})}),
 \label{4.21}
\end{equation}
where
\begin{equation}
\Lambda(x)=\left(
\begin{array}
[c]{cccc}%
\frac{\partial\lambda}{\partial x_{1}} &
\frac{\partial\omega_{1}}{\partial
x_{1}} & ... & \frac{\partial\omega_{N-1}}{\partial x_{1}}\\
... & ... & ... & ...\\
\frac{\partial\lambda}{\partial x_{N}} &
\frac{\partial\omega_{1}}{\partial
x_{N}} & ... & \frac{\partial\omega_{N-1}}{\partial x_{N}}%
\end{array}
\right),
 \label{a4.21}
\end{equation}
and similarly for the variables $y$. Then in the variables $(x,t)$

\[
(\nabla_{\rho}v^{\pm},\nabla_{\rho}\lambda_{x})=
(\mathcal{E}_{\rho}\Lambda\nabla_{(\lambda,\omega)}v^{\pm},
\mathcal{E}_{\rho}\Lambda\nabla_{(\lambda,\omega)}\lambda_{x}).
\]
Note that
$\nabla_{(\lambda_{x},\omega_{x})}\lambda_{x}=\{1,0,...,0\}$, and
also  $v^{\pm}\equiv 0$ on $\Gamma$, hence $\partial
v^{\pm}/\partial \omega_{i}=0$, and therefore
\[
\nabla_{(\lambda_{x},\omega_{x})}v^{\pm}=\{\frac{\partial
v^{\pm}}{\partial \lambda_{x}},0,...,0\}=\frac{\partial
v^{\pm}}{\partial \lambda_{x}}\{1,0,...,0\}=\frac{\partial
v^{\pm}}{\partial
\lambda_{x}}\nabla_{(\lambda_{x},\omega_{x})}\lambda_{x}.
\]
Thus we obtain

\begin{equation}
(\nabla_{y}v^{\pm},\nabla_{y}\Phi_{\rho})\circ
e_{\rho}(x,t)=(\nabla_{\rho}v^{\pm},\nabla_{\rho}\lambda_{x})=
\frac{\partial v^{\pm}}{\partial
\lambda_{x}}(\nabla_{\rho}\lambda_{x}, \nabla_{\rho}\lambda_{x}).
 \label{4.22}
\end{equation}
On the other hand, due to the definition of $\Phi_{\rho}(y,\tau)$,

\begin{equation}
(\nabla_{\rho}\lambda_{x},
\nabla_{\rho}\lambda_{x})=(\nabla_{y}(\lambda_{x}\circ
e_{\rho}^{-1}),\nabla_{y}(\lambda_{x}\circ e_{\rho}^{-1}))\circ
e_{\rho}=(\nabla_{y}\Phi_{\rho},\nabla_{y}\Phi_{\rho})\circ
e_{\rho}.
 \label{a4.22}
\end{equation}
Using introduced in \eqref{4.21} matrix $\Lambda(y)$, we have

\[
(\nabla_{y}\Phi_{\rho},\nabla_{y}\Phi_{\rho})=
(\Lambda(y)\nabla_{(\lambda_{y},\omega_{y})}\Phi_{\rho},\Lambda(y)\nabla_{(\lambda_{y},\omega_{y})}\Phi_{\rho})=
\]
\begin{equation}
=(\nabla_{(\lambda_{y},\omega_{y})}\Phi_{\rho},\Lambda(y)^{*}\Lambda(y)\nabla_{(\lambda_{y},\omega_{y})}\Phi_{\rho}).
 \label{4.23}
\end{equation}
First, by the definition of $\Phi_{\rho}$,

\[
\frac{\partial \Phi_{\rho}}{\partial
\lambda_{y}}=\frac{\partial}{\partial
\lambda_{y}}(\lambda_{y}-\rho(y,\tau))=1-\rho_{\lambda_{y}},
\]
\begin{equation}
\frac{\partial \Phi_{\rho}}{\partial
\omega_{yi}}=\frac{\partial}{\partial
\omega_{yi}}(\lambda_{y}-\rho(y,\tau))=-\rho_{\omega_{yi}}.
 \label{4.24}
\end{equation}
In addition, since the coordinate $\lambda_{y}$ is counted by the
normal to $\Gamma$, and $\omega_{yi}$ are coordinates on the surface
$\Gamma$, then

\[
(\nabla_{y}\lambda(y),\nabla_{y}\lambda(y))=1, \ \
(\nabla_{y}\lambda(y),\nabla_{y}\omega_{i}(y))=0, \ i=1,...,N-1.
\]
Therefore the matrix $\Lambda^{*}(y)\Lambda(y)$ has the form
\begin{equation}
\Lambda^{\ast}(y)\Lambda(y)=\left(
\begin{array}
[c]{ccccc}%
1 & 0 & 0 & ... & 0\\
0 & m_{11} & m_{12} & ... & m_{1(N-1)}\\
... & ... & ... & ... & ...\\
0 & m_{(N-1)1} & m_{(N-1)2} & ... & m_{(N-1)(N-1)}%
\end{array}
\right)  , \label{4.25}%
\end{equation}
where
\begin{equation}
m_{ij}=m_{ji}=(\nabla_{y}\omega_{i}(y),\nabla_{y}\omega_{j}(y))- \label{4.26}%
\end{equation}
are some smooth functions.

Thus,
\[
(\nabla_{(\lambda_{y},\omega_{y})}\Phi_{\rho},\Lambda^{\ast}(y)\Lambda
(y)\nabla_{(\lambda_{y},\omega_{y})}\Phi_{\rho})=
\]%
\begin{equation}
=(1-\rho_{\lambda_{y}})^{2}+\sum\limits_{i,j=1}^{N-1}m_{ij}(y)\rho
_{\omega_{yi}}\rho_{\omega_{yj}}. \label{4.27}%
\end{equation}

Make now in \eqref{4.27} the change of variables \eqref{4.1}, and
recalculate the derivatives of $\rho$ with respect to
$(\lambda_{y},\omega_{y})$ in terms of the derivatives with respect
to $(\lambda_{x},\omega_{x})$. We have

\begin{equation}
\rho_{\lambda_{y}}\circ e_{\rho}=\rho_{t}\frac{\partial
t}{\partial\lambda
_{y}}+\rho_{\lambda_{x}}\frac{\partial\lambda_{x}}{\partial\lambda_{y}}%
+\sum\limits_{i=1}^{N-1}\rho_{\omega_{xi}}\frac{\partial\omega_{xi}}%
{\partial\lambda_{y}}. \label{4.28}%
\end{equation}
It follows from the definition of the mapping $e_{\rho}$ that

\begin{equation}
\frac{\partial t}{\partial\lambda_{y}}=0,\quad\frac{\partial\omega_{xi}%
}{\partial\lambda_{y}}=0. \label{4.29}%
\end{equation}
At the same time by \eqref{4.28}, \eqref{4.29}

\[
\frac{\partial\lambda_{x}}{\partial\lambda_{y}}=1-\rho_{\lambda_{y}}%
=1-\rho_{\lambda_{x}}\frac{\partial\lambda_{x}}{\partial\lambda_{y}},
\]
that is

\begin{equation}
\frac{\partial\lambda_{x}}{\partial\lambda_{y}}=\frac{1}{1+\rho_{\lambda_{x}}%
}. \label{4.30}%
\end{equation}
Therefore by \eqref{4.28}, \eqref{4.29} and \eqref{4.30}

\begin{equation}
\rho_{\lambda_{y}}\circ
e_{\rho}=\frac{\rho_{\lambda_{x}}}{1+\rho_{\lambda
_{x}}}. \label{4.31}%
\end{equation}
Further,

\begin{equation}
\rho_{\omega_{yi}}\circ e_{\rho}=\rho_{t}\frac{\partial
t}{\partial\omega
_{yi}}+\rho_{\lambda_{x}}\frac{\partial\lambda_{x}}{\partial\omega_{yi}}%
+\sum\limits_{j=1}^{N-1}\rho_{\omega_{xi}}\frac{\partial\omega_{xj}}%
{\partial\omega_{yi}}, \label{4.32}%
\end{equation}
and

\begin{equation}
\frac{\partial t}{\partial\omega_{yi}}=0,\quad\frac{\partial\omega_{xj}%
}{\partial\omega_{yi}}=\delta_{ij},\quad i,j=1,...,N-1. \label{4.33}%
\end{equation}
At the same time

\[
\frac{\partial(\lambda_{x}\circ
e_{\rho})}{\partial\omega_{yi}}=\left[
\frac{\partial}{\partial\omega_{yi}}(\lambda_{y}-\rho(y,\tau))\right]
\circ e_{\rho}=-\rho_{\omega_{yi}}\circ e_{\rho},
\]
That is by virtue of \eqref{4.32} and \eqref{4.33},
\begin{equation}
\rho_{\omega_{yi}}\circ
e_{\rho}=\rho_{\lambda_{x}}(-\rho_{\omega_{yi}}\circ
e_{\rho})+\rho_{\omega_{xi}}, \label{4.34}%
\end{equation}
hence by \eqref{4.34},
\begin{equation}
\rho_{\omega_{yi}}\circ
e_{\rho}=\frac{\rho_{\omega_{xi}}}{1+\rho_{\lambda
_{x}}}. \label{4.35}%
\end{equation}

Thus, it follows from  \eqref{4.22}, \eqref{4.27}, \eqref{4.31} and
\eqref{4.35} that in \eqref{4.22}
\begin{equation}
(\nabla_{\rho}\lambda_{x},\nabla_{\rho}\lambda_{x})=\frac{1}{(1+\rho
_{\lambda_{x}})^{2}}\left[
1+\sum\limits_{i,j=1}^{N-1}m_{ij}(x,\rho
)\rho_{\omega_{xi}}\rho_{\omega_{xj}}\right]  . \label{4.36}%
\end{equation}

Finally, the relation \eqref{4.11} follows from the relations
\eqref{4.18}, \eqref{4.19}, \eqref{4.22} and \eqref{4.36}.

\section{The linearization of the problem. }\label{s5}

Our goal in this section is  the extraction of the principal linear
part of the problem \eqref{4.9}-\eqref{a4.1} in terms of the
deviation of the unknown functions $(v^{+},v^{-},\rho)$ from
functions constructed from the initial data and satisfying
\eqref{4.9}-\eqref{a4.1} for $t=0$, as it was done in \cite{2},
\cite{5}.

Note that from the equations \eqref{4.9}, \eqref{4.11} and from the
initial data \eqref{4.13} we can calculate the derivatives with
respect to time $\partial v^{\pm}/\partial t$ and $\partial
\rho/\partial t$ at $t=0$:
\begin{equation}
\frac{\partial\rho}{\partial t}(\omega,0)=\rho_{1}(\omega)\equiv\frac{1}%
{k}(a^{+}\frac{\partial
v_{0}^{+}}{\partial\lambda}-a^{-}\frac{\partial
v_{0}^{-}}{\partial\lambda})|_{\Gamma}, \label{5.1}%
\end{equation}%

\begin{equation}
\frac{\partial v^{\pm}}{\partial
t}(x,0)=v_{1}^{\pm}(x)\equiv\frac{\partial
v_{0}^{\pm}}{\partial\lambda}\rho_{1}+a^{\pm}\left|
v_{0}^{\pm}(x)\right|  ^{\alpha
}\nabla^{2}v_{0}^{\pm}(x), \label{5.2}%
\end{equation}
and, in view of the assumptions \eqref{1.10}, \eqref{1.11},
\begin{equation}
\rho_{1}(\omega)\in C^{1+\beta'-\alpha}(\Gamma), \beta'  \equiv \gamma'(1-\alpha/2)>\beta, \quad v_{1}^{\pm}(x)\in
C^{\gamma'}_{s}(\overline{\Omega^{\pm}}). \label{5.3}%
\end{equation}

Completely analogous to \cite{12}, Ch.IV, on the  base of results
\cite{12.1} on the solvability of the Cauchy-Dirichlet problem for
degenerate equations we construct such functions
$w^{\pm}(x,t)\in C_{s}^{2+\gamma^{\prime},1+\gamma^{\prime}%
/2}(\overline{\Omega_{T}^{\pm}})$ that

\begin{equation}
|  w^{\pm}| _{s,\overline{\Omega_{T}^{\pm}}}^{(2+\gamma
^{\prime},1+\gamma^{\prime}/2)}\leq C\left(  \left|
v_{0}^{\pm}\right| _{s,\overline{\Omega^{\pm}}}^{(2+\gamma')}+\left|
v_{1}^{\pm}\right| _{s,\overline{\Omega^{\pm}}}^{(\gamma')}\right)
\leq C\left|  v_{0}^{\pm
}\right|  _{s,\overline{\Omega^{\pm}}}^{(2+\gamma')} \label{5.4}%
\end{equation}
and

\begin{equation}
w^{\pm}(x,0)=v_{0}^{\pm}(x),\frac{\partial w^{\pm}}{\partial t}(x,0)=v_{1}%
^{\pm}(x),w^{\pm}(x,t)|_{\Gamma_{T}}=0,w^{\pm}(x,t)|_{\Gamma_{T}^{\pm}}%
=h^{\pm}. \label{5.5}%
\end{equation}
In addition, just as described in \cite{12}, Chapter 4, there is a
such function  $\sigma(\omega,t)\in C^{3+\beta'-\alpha,1+\frac{1+\beta'-\alpha}{2}%
}(\Gamma_{T})$ that

\begin{equation}
\left|  \sigma\right|
_{\Gamma_{T}}^{(3+\beta'-\alpha,\frac{3+\beta'-\alpha}{2})}\leq C(\left|
v_{0}^{+}\right|
_{s,\overline{\Omega^{+}}}^{(2+\gamma')}+\left|
v_{0}^{-}\right|  _{s,\overline{\Omega^{-}}}^{(2+\gamma')}) \label{5.6}%
\end{equation}
and

\begin{equation}
\sigma(\omega,0)=\rho(\omega,0)=0,\quad\frac{\partial\sigma}{\partial
t}(\omega,0)=\rho_{1}(\omega). \label{5.7}%
\end{equation}
Moreover, by the method described in \cite{12}, Ch.IV, the function
$\sigma(\omega,t)$ can be extended with the class and with the
inequality \eqref{5.6} to a function defined in
$\overline{\Omega_{T}}$ which is non-zero only in the neighborhood
$\mathcal{N} \times [0,T]$ of the surface $\Gamma_{T}$.

The linearization of the relations \eqref{4.9}-\eqref{a4.1} consists
in the following (we describe the general scheme of the arguments -
the exact formulations will be given below). We denote the space

\begin{equation}
P^{2+\beta-\alpha}(\Gamma_{T})=\{\rho: \rho\in
C^{2+\beta-\alpha,\frac{2+\beta-\alpha}{2-\alpha}}(\Gamma_{T}), \
\rho_{t}\in
C^{1+\beta-\alpha,\frac{1+\beta-\alpha}{2-\alpha}}(\Gamma_{T})\}
\label{Dop.1}
\end{equation}
with the norm

\begin{equation}
|\rho|_{P^{2+\beta-\alpha}(\Gamma_{T})}\equiv
|\rho|_{C^{2+\beta-\alpha,\frac{2+\beta-\alpha}{2-\alpha}}(\Gamma_{T})}+
|\rho_{t}|_{C^{1+\beta-\alpha,\frac{1+\beta-\alpha}{2-\alpha}}(\Gamma_{T})}.
\label{Dop.2}
\end{equation}

Denote also

\begin{equation}
\psi=(v^{+},v^{-},\rho)\in \mathcal{H}\equiv C_{s}^{2+\gamma,1+\gamma/2}%
(\overline{\Omega_{T}^{+}})\times C_{s}^{2+\gamma,1+\gamma/2}%
(\overline{\Omega_{T}^{-}})\times P^{2+\beta-\alpha}(\Gamma_{T}),
\label{5.9}%
\end{equation}
\[
\psi_{0}=(w^{+},w^{-},\sigma),
\]
and represent the the relations \eqref{4.9}-\eqref{a4.1} as

\begin{equation}
F(\psi)=0 \label{5.10}%
\end{equation}
with some non-linear operator of $\psi$. Keeping essentially in mind
the application of Newton's method, we represent the relation
\eqref{5.10} as

\[
F^{\prime}(\psi_{0})(\psi-\psi_{0})=-F(\psi_{0})+[F(\psi_{0})+F^{\prime}%
(\psi_{0})(\psi-\psi_{0})-F(\psi)]\equiv
\]%
\begin{equation}
\equiv f_{0}+G(\psi-\psi_{0}), \label{5.11}%
\end{equation}
where $F'(\psi_{0})$ is the  Frechet derivative of $F(\psi)$ at the
point $\psi_{0}$. In this case, as the new unknown we consider the
difference

\begin{equation}
\varphi=\psi-\psi_{0}=(v^{+}-w^{+},v^{-}-w^{-},\rho-\sigma),
\label{5.12}
\end{equation}
which belongs to the spaces with zero, that is,

\begin{equation}
\varphi\in \mathcal{H}_{0}\equiv
C_{0,s}^{2+\gamma,1+\gamma/2}(\overline{\Omega
_{T}^{+}})\times C_{0,s}^{2+\gamma,1+\gamma/2}(\overline{\Omega_{T}^{-}%
})\times P_{0}^{2+\beta-\alpha}(\Gamma_{T}). \label{5.13}%
\end{equation}
By the construction of the element $\psi_{0}=(w^{+},w^{-},\sigma)$,
it has an increased smoothness ($\gamma'>\gamma$) and satisfies the
relation $F(\psi_{0})=0$ for $t=0$. Therefore, using the
inequalities \eqref{B.3}-\eqref{B.6}, we can estimate

\begin{equation}
\left\|  f_{0}\right\|  =\left\|  -F(\psi_{0})\right\|  \leq
CT^{\mu}. \label{5.14}%
\end{equation}

Below we show that the operator $F'(\psi_{0})$ has the bounded
inverse in the appropriate spaces, so that the equation \eqref{5.11}
can be rewritten as

\[
\varphi=[F^{\prime}(\psi_{0})]^{-1}f_{0}+[F^{\prime}(\psi_{0})]^{-1}%
G(\varphi)\equiv
\]%
\begin{equation}
\equiv h_{0}+H(\varphi)\equiv K(\varphi), \label{5.15}%
\end{equation}
where by \eqref{5.14}

\begin{equation}
\left\|  h_{0}\right\|  \leq CT^{\mu}, \label{5.16}%
\end{equation}
and the operator $H(\varphi)$ is the "quadratic" with respect to
$\varphi$ by the smoothness of $F(\psi)$ in its argument and by the
definition of $G(\psi-\psi_{0})=G(\varphi)$ in \eqref{5.11}:

\begin{equation}
\left\|  H(\varphi)\right\|  \leq C\left\|  \varphi\right\|
^{2},\left\| H(\varphi_{2})-H(\varphi_{1})\right\|  \leq C(\left\|
\varphi_{1}\right\| +\left\|  \varphi_{2}\right\|  )\left\|
\varphi_{2}-\varphi_{1}\right\|  .
\label{5.17}%
\end{equation}
For sufficiently small $T>0$ it follows from  \eqref{5.16} and
\eqref{5.17} that the operator $K(\varphi)$ maps some small ball
$B_{r}\subset \mathcal{H}_{0}$ with a small $r$ into itself and
$K(\varphi)$ is a contractive there.  The only fixed point of this
operator gives, obviously, the solution of the original problem.

Thus, our goal now is to write the problem \eqref{4.9}-\eqref{a4.1}
as \eqref{5.11}.

Denote

\bigskip%
\begin{equation}
\theta^{\pm}=v^{\pm}-w^{\pm},\quad\delta=\rho-\sigma. \label{5.18}%
\end{equation}

\begin{lemma} \label{L5.1}
The problem \eqref{4.9}-\eqref{a4.1} can be represented as a problem
for the unknown functions $\theta^{\pm}$ and $\delta$ as follows

\[
\frac{\partial\theta^{\pm}}{\partial t}-\left|  u_{0}^{\pm}\right|
^{\alpha }\nabla^{2}\theta^{\pm}-\frac{\partial
w^{\pm}}{\partial\lambda}\left(
\frac{\partial\delta}{\partial t}-\left|  u_{0}^{\pm}\right|  ^{\alpha}%
\nabla^{2}\delta\right)  =
\]%

\begin{equation}
=F_{1}^{\pm}(x,t;\theta,\delta)+F_{2}^{\pm}(x,t;\theta,\delta)\equiv
F_{1}^{\pm}(x,t;\varphi)+F_{2}^{\pm}(x,t;\varphi),(x,t)\in\Omega_{T}^{\pm},
\label{5.19}%
\end{equation}%

\begin{equation}
\theta^{+}=\theta^{+}=0,\quad(x,t)\in\Gamma_{T}, \label{5.20}%
\end{equation}%

\[
k\delta_{t}+[a^{+}\frac{\partial\theta^{+}}{\partial\lambda}-a^{-}%
\frac{\partial\theta^{-}}{\partial\lambda}]-\delta_{\lambda}[a^{+}%
\frac{\partial w^{+}}{\partial\lambda}-a^{-}\frac{\partial w^{-}}%
{\partial\lambda}]=
\]%

\begin{equation}
=F_{3}(x,t;\varphi)+F_{4}(x,t;\varphi),\quad(x,t)\in\Gamma_{T}, \label{5.21}%
\end{equation}%

\begin{equation}
\theta^{\pm}=0,\quad(x,t)\in\Gamma_{T}^{\pm}, \label{5.22}%
\end{equation}%

\begin{equation}
\theta^{\pm}(x,0)=0,\quad\delta(\omega,0)=0, \label{5.23}%
\end{equation}%

\begin{equation}
\delta(x,t)=E\delta(\omega,t), \label{a5.23}%
\end{equation}
where for arbitrary $\varphi=(\theta^{+},\theta^{-},\delta)\in
B_{r}\subset \mathcal{H}_{0}$, $r<\gamma_{0}/2$ in the righthand
sides $F_{i}$ of the relations \eqref{5.19}- \eqref{5.23} all the
functions $F_{i}$ vanish at $t=0$ and the following estimates are
valid

\begin{equation}
|  F_{1}^{\pm}(x,t;\varphi)|  _{s,\overline{\Omega_{T}^{\pm}}%
}^{(\gamma)}\leq CT^{\mu}, \label{5.24}%
\end{equation}%

\begin{equation}
| F_{1}^{\pm}(x,t;\varphi_{2})-F_{1}^{\pm}(x,t;\varphi_{1})|
_{s,\overline{\Omega_{T}^{\pm}}}^{(\gamma)}\leq CT^{\mu}\left\|
\varphi_{2}-\varphi_{1}\right\|  _{\mathcal{H}}, \label{5.25}%
\end{equation}%

\begin{equation}
| F_{2}^{\pm}(x,t;\varphi)|  _{s,\overline{\Omega_{T}^{\pm}}%
}^{(\gamma)}\leq C\left\|  \varphi\right\|  _{\mathcal{H}}^{2}, \label{5.26}%
\end{equation}%

\begin{equation}
| F_{2}^{\pm}(x,t;\varphi_{2})-F_{2}^{\pm}(x,t;\varphi_{1})|
_{s,\overline{\Omega_{T}^{\pm}}}^{(\gamma)}\leq C(\left\|
\varphi_{2}\right\|  _{\mathcal{H}}+\left\| \varphi_{1}\right\|
_{\mathcal{H}})\left\|
\varphi_{2}-\varphi_{1}\right\|  _{\mathcal{H}}, \label{5.27}%
\end{equation}%

\begin{equation}
|  F_{3}(x,t;\varphi)| _{\Gamma_{T}}^{(1+\beta
-\alpha,\frac{1+\beta
-\alpha}{2-\alpha})}\leq CT^{\mu}, \label{5.28}%
\end{equation}%

\begin{equation}
| F_{3}(x,t;\varphi_{2})-F_{3}^{\pm}(x,t;\varphi_{1})|
_{\Gamma_{T}}^{(1+\beta-\alpha,\frac{1+\beta
-\alpha}{2-\alpha})}\leq CT^{\mu}\left\| \varphi
_{2}-\varphi_{1}\right\|  _{\mathcal{H}}, \label{5.29}%
\end{equation}%

\begin{equation}
|  F_{4}(x,t;\varphi)|  _{\Gamma_{T}}^{(1+\beta
-\alpha,\frac{1+\beta
-\alpha}{2-\alpha})}\leq C\left\|  \varphi\right\|  _{\mathcal{H}}^{2}, \label{5.30}%
\end{equation}%

\begin{equation}
| F_{4}(x,t;\varphi_{2})-F_{4}(x,t;\varphi_{1})|
_{\Gamma_{T}}^{(1+\beta-\alpha,\frac{1+\beta
-\alpha}{2-\alpha})}\leq C(\left\| \varphi _{2}\right\|
_{\mathcal{H}}+\left\| \varphi_{1}\right\| _{\mathcal{H}})\left\|
\varphi
_{2}-\varphi_{1}\right\|  _{\mathcal{H}}, \label{5.31}%
\end{equation}

\end{lemma}

\begin{proof}

Meaning of the inequalities \eqref{5.24}-\eqref{5.31} is that,
according to \eqref{5.11}, the expressions $F_{1}^{\pm}$ and $F_{3}$
contain smoother terms and to evaluate them, we use inequalities
\eqref{B.3}-\eqref{B.6}, and the expressions $F_{2}^{\pm}$ and
$F_{4}$ are "quadratic" with respect to $\varphi$.

In the case of a uniformly parabolic equation in
\eqref{4.3}(\eqref{1.3}), this lemma is proved in details in
\cite{5}, Section 2.3. Therefore, we mention only the differences
that arise in the case of degenerate equations.

First, in contrast to the \cite{5}, we can not expect that the
extended function $\rho(x,t)=E\rho(\omega,t) $ satisfies the
condition $\partial \rho(x,t)/\partial \lambda = 0$ on $\Gamma_T$,
and so we explain the obtaining of the relation \eqref{5.21} from
the relation \eqref{4.11}. The relation \eqref{5.21} is obtained
from \eqref{4.11} explicitly after substitution in \eqref{4.11} the
expressions $v^{\pm}= \theta^{\pm}$, $\rho=\delta + \sigma$ and the
transfer of the junior and quadratic terms in the righthand part. It
is easy to verify that \eqref{5.21} coincides with \eqref{4.11} for

\[
F_{3}(x,t,\varphi)\equiv\left\{  \left[  \sum\limits_{i,j=1}^{N-1}%
m_{ij}(x,\rho)\rho_{\omega_{i}}\rho_{\omega_{j}}(a^{+}\frac{\partial v^{+}%
}{\partial\lambda}-a^{-}\frac{\partial
v^{-}}{\partial\lambda})\right] \right.  -
\]%

\[
-\left[  k\sigma_{t}-(a^{+}\frac{\partial w^{+}}{\partial\lambda}-a^{-}%
\frac{\partial w^{-}}{\partial\lambda})\right]
-\delta_{\lambda}\left[ k\sigma_{t}-(a^{+}\frac{\partial
w^{+}}{\partial\lambda}-a^{-}\frac{\partial
w^{-}}{\partial\lambda})\right]  -
\]%

\begin{equation}
\left.  -\left[  k\rho_{t}\sigma_{\lambda}\right]  \right\}  , \label{5.32}%
\end{equation}%

\begin{equation}
F_{4}(x,t,\varphi)\equiv-k\delta_{t}\delta_{\lambda}, \label{5.33}%
\end{equation}
where in \eqref{5.32} $\rho=\delta+\sigma$, $v^{\pm}=\theta^{\pm}$.

Using the fact that
$\delta(\omega,0)=\sigma(\omega,0)=\rho(\omega,0)=0$, estimating
each term in square brackets in \eqref{5.32} separately, and using
the inequalities \eqref{B.3}-\eqref{B.6} it is easy to obtain for
$F_{3}(x,t,\varphi)$ the estimates \eqref{5.28}, \eqref{5.29}. For
example, since by the construction

\[
k\sigma_{t}(x,0)-(a^{+}\frac{\partial w^{+}(x,0)}{\partial\lambda}-a^{-}%
\frac{\partial w^{-}(x,0)}{\partial\lambda})=0,\quad x\in\Gamma,
\]
then

\[
\left|  \delta_{\lambda}\left[  k\sigma_{t}-(a^{+}\frac{\partial w^{+}%
}{\partial\lambda}-a^{-}\frac{\partial
w^{-}}{\partial\lambda})\right] \right|
_{\Gamma_{T}}^{(1+\beta-\alpha,\frac{1+\beta
-\alpha}{2-\alpha})}\leq
\]%

\[
\leq CT^{\mu}\left|  \delta_{\lambda}\right|  _{\Gamma_{T}%
}^{(1+\beta-\alpha,\frac{1+\beta -\alpha}{2-\alpha})}\left| \left[
k\sigma_{t}-(a^{+}\frac{\partial
w^{+}}{\partial\lambda}-a^{-}\frac{\partial
w^{-}}{\partial\lambda})\right] \right|
_{\Gamma_{T}}^{(1+\beta-\alpha,\frac{1+\beta
-\alpha}{2-\alpha})}\leq
\]%

\[
\leq CT^{\mu}\left|  \delta\right| _{\Gamma_{T}}^{(2+\beta
-\alpha,\frac{2+\beta -\alpha}{2-\alpha})}\leq CT^{\mu}\left\|
\varphi\right\|  _{\mathcal{H}}.
\]
Since this term is linear with respect to $\varphi$, this yields \eqref{5.28}, \eqref{5.29} for this term.
The remaining terms in the definition of $F_{3}(x,t,\varphi)$ are treated similarly.

As for $F_{4}(x,t,\varphi)$, the estimates \eqref{5.30}, \eqref{5.31} for this expression are obvious
because it is quadratic.

Another difference from \cite{5} is the presence of a degenerate factor in the third (elliptical) terms
in the left-hand side of \eqref{4.9}. Represent this term as (we consider only the equation for the sign "+")

\[
(v^{+})^{\alpha}\nabla_{\rho}^{2}v^{+}=(v_{0}^{+})^{\alpha}\nabla_{\rho}%
^{2}v^{+}+[(v^{+})^{\alpha}-(v_{0}^{+})^{\alpha}]\nabla_{\rho}^{2}v^{+}\equiv
\]%

\begin{equation}
\equiv(v_{0}^{+})^{\alpha}\nabla_{\rho}^{2}v^{+}+A_{2}^{(1)+}(\varphi).
\label{5.38}%
\end{equation}
and show that $A_{2}^{(1)+}(\varphi)$ satisfies the inequalities
\begin{equation}
\left|  A_{2}^{(1)+}(\varphi)\right|  _{s,\overline{\Omega_{T}^{+}}%
}^{(\gamma)}\leq CT^{\mu},\left|  A_{2}^{(1)+}%
(\varphi_{2})-A_{2}^{(1)+}(\varphi_{1})\right|  _{s,\overline{\Omega_{T}%
^{+}}}^{(\gamma)}\leq CT^{\mu}\left\| \varphi
_{2}-\varphi_{1}\right\|  _{\mathcal{H}}, \label{5.36}%
\end{equation}%

Let us assume that the function $\lambda=\lambda(x)$ is extended with the preservation of the
class from the neighborhood $\mathcal{N}$
of the surface $\Gamma$ on all $\overline{\Omega}$  to a function
satisfying the conditions

\[
\nu d^{+}(x)\leq\lambda(x)\leq\nu^{-1}d^{+}(x),
\]
retaining for her the same notation. Write further $A_{2}^{(1)+}(\varphi) $in the form

\begin{equation}
A_{2}^{(1)+}(\varphi)=\left[\left(  \frac{v^{+}}{\lambda}\right)
^{\alpha }-\left(  \frac{v_{0}^{+}}{\lambda}\right)
^{\alpha}\right]\lambda^{\alpha
}(x)\nabla_{\rho}^{2}v^{+}. \label{5.39}%
\end{equation}
Since $v_{0}, v^{+}=0$ for $x\in \Gamma$, then for $x\in
\mathcal{N}$

\begin{equation}
\frac{v^{+}}{\lambda}=\int\limits_{0}^{1}\frac{\partial
v^{+}}{\partial \lambda}(\lambda
s,\omega,t)ds,\quad\frac{v_{0}^{+}}{\lambda}=\int
\limits_{0}^{1}\frac{\partial v_{0}^{+}}{\partial\lambda}(\lambda
s,\omega)ds,
\label{5.40}%
\end{equation}
where we assume the neighborhood $\mathcal{N}$ so small that

\begin{equation}
\frac{\partial v_{0}^{+}}{\partial\lambda}\geq\nu>0,\quad x\in
\mathcal{N},
\label{5.41}%
\end{equation}
In addition, we assume that $T$ is so small that

\[
\frac{\partial w^{+}}{\partial\lambda}(x,t)\geq\nu>0,\quad(x,t)\in
\mathcal{N}\times\lbrack0,T].
\]
Assuming now that the radius $r=r(\nu)$ of the ball $\mathcal{B}_{r}\ni \varphi$
is sufficiently small, we can assume that for $\varphi \in \mathcal{B}_{r}$

\begin{equation}
\frac{\partial v^{+}}{\partial\lambda}(x,t)=\frac{\partial w^{+}}%
{\partial\lambda}(x,t)+\frac{\partial\theta^{+}}{\partial\lambda}(x,t)\geq
\nu>0,\quad(x,t)\in N\times\lbrack0,T]. \label{5.42}%
\end{equation}
In addition, outside the neighborhood $\mathcal{N}_{T}$ holds

\begin{equation}
v_{0}^{+}(x)\geq\nu>0,\quad x\in\overline{\Omega^{+}}\setminus \mathcal{N}. \label{5.43}%
\end{equation}
Therefore, assuming as above $T$ and $r$ sufficiently small, we can assume that
\begin{equation}
v^{+}(x,t)=w^{+}(x,t)+\theta^{+}(x,t)\geq\nu>0,\quad x\in\overline{\Omega^{+}%
}\setminus N. \label{5.44}%
\end{equation}
Thus, we have the representation
\begin{equation}
\frac{v^{+}(x,t)}{\lambda(x)}=\Phi^{+}(x,t,\varphi)=%
\genfrac{\{}{.}{0pt}{}{\int\limits_{0}^{1}\frac{\partial
v^{+}}{\partial \lambda}(\lambda s,\omega,t)ds\geq\nu,\quad x\in
\mathcal{N},}{\frac{v^{+}}{\lambda
},\quad x\in\overline{\Omega^{+}}\setminus \mathcal{N},}%
\label{5.45}%
\end{equation}
and
\begin{equation}
\left|  \Phi^{+}(x,t,\varphi)\right|_{s,\overline{\Omega_{T}^{+}}}%
^{(\gamma)}\leq C(\left|  \frac{\partial v^{+}}{\partial\lambda
}\right| _{s,\overline{\Omega_{T}^{+}}}^{(\gamma)}+\left|
v^{+}\right|_{s,\overline{\Omega_{T}^{+}}}^{(\gamma)})\leq C\left(
\left\|  \varphi\right\|  _{\mathcal{H}}\right)  , \label{5.46}%
\end{equation}
and also
\begin{equation}
\Phi^{+}(x,t,\varphi)\geq\nu>0. \label{5.47}%
\end{equation}
Similarly
\begin{equation}
\frac{v_{0}^{+}(x)}{\lambda(x)}=\Phi_{0}^{+}(x)=%
\genfrac{\{}{.}{0pt}{}{\int\limits_{0}^{1}\frac{\partial v_{0}^{+}}%
{\partial\lambda}(\lambda s,\omega)ds\geq\nu,\quad x\in \mathcal{N},}{\frac{v_{0}^{+}%
}{\lambda},\quad x\in\overline{\Omega^{+}}\setminus
\mathcal{N},}\label{5.48}
\end{equation}
 so
\begin{equation}
\left|  \Phi_{0}^{+}(x)\right|  _{s,\overline{\Omega_{T}^{+}}}%
^{(\gamma)}\leq C,\quad\Phi_{0}^{+}(x)\geq\nu>0.\label{5.49}
\end{equation}
In addition, it follows from the representations \eqref{5.45} and \eqref{5.48} that

\[
\left|  \Phi^{+}(x,t,\varphi)-\Phi_{0}^{+}(x)\right| _{s,\overline
{\Omega_{T}^{+}}}^{(\gamma)}\leq
\]%
\begin{equation}
\leq C\left(  \left|  v^{+}-v_{0}^{+}\right|_{s,\overline{\Omega_{T}^{+}}%
}^{(\gamma)}+\left| \frac{\partial}{\partial\lambda}\left(
v^{+}-v_{0}^{+}\right) \right|
_{s,\overline{\Omega_{T}^{+}}}^{(\gamma)}\right)  \leq C(\left\|
\varphi\right\|  )T^{\mu}, \label{5.50}%
\end{equation}
and similarly
\begin{equation}
\left| \Phi^{+}(x,t,\varphi_{1})-\Phi^{+}(x,t,\varphi_{2})\right|
_{s,\overline{\Omega_{T}^{+}}}^{(\gamma)}\leq CT^{\mu}\left\|
\varphi_{2}-\varphi_{1}\right\|  . \label{5.51}%
\end{equation}
By the properties \eqref{5.47} and \eqref{5.49}, the mapping
\begin{equation}
\varphi\rightarrow B_{1}(\varphi)\equiv\left(
\frac{v^{+}}{\lambda}\right) ^{\alpha}-\left(
\frac{v_{0}^{+}}{\lambda}\right)  ^{\alpha}=\left[  \Phi
^{+}(x,t,\varphi)\right]  ^{\alpha}-\left[
\Phi_{0}^{+}(x,\varphi)\right]
^{\alpha} \label{a5.51}%
\end{equation}
is smooth, and by \eqref{5.50}, \eqref{5.51}
\[
\left|  B_{1}(\varphi)\right|
_{s,\overline{\Omega_{T}^{+}}}^{(\gamma)}\leq CT^{\mu},
\]%
\begin{equation}
\left|  B_{1}(\varphi_{2})-B_{1}(\varphi_{1})\right|
_{s,\overline{\Omega _{T}^{+}}}^{(\gamma)}\leq CT^{\mu}\left\|
\varphi_{2}-\varphi
_{1}\right\|  _{\mathcal{H}}. \label{5.52}%
\end{equation}

Now from the definition of the expression $A_{2}^{(1)+}(\varphi)$ in
 \eqref{5.39},from the relations \eqref{5.52} and from the smooth dependence of
$\nabla_{\rho}=\mathcal{E}_{\rho}\nabla$ on $\rho$, taking into account
that the factor $\lambda^{\alpha}(x)$ is appropriate  for the weighted
 estimates of the second derivatives of $v^{+}$ and $\theta^{+}$ in the
space $C_{0,s}^{2+\gamma, 1+\gamma/2}(\overline{\Omega^{+}_{T}})$, it is
easy to see that

\[
\left|  A_{2}^{(1)+}(\varphi)\right|_{s,\overline{\Omega_{T}^{+}}}%
^{(\gamma)}\leq CT^{\mu},
\]%
\begin{equation}
\left| A_{2}^{(1)+}(\varphi_{2})-A_{2}^{(1)+}(\varphi_{1})\right|
_{s,\overline{\Omega_{T}^{+}}}^{(\gamma)}\leq CT^{\mu}\left\|
\varphi_{2}-\varphi_{1}\right\|  _{\mathcal{H}}. \label{5.53}%
\end{equation}

Thus, in view of \eqref{5.38} and \eqref{5.53} the linearization of the equation
\eqref{4.9} is reduced to the linearization of a linear on $v^{+}$ equation that
was done in details in \cite{5}.

Note also that the insignificant difference between \eqref{5.19} from \cite{5} is still in that
we, in fact, leave in the left-hand side of \eqref{5.19} only the leading terms, moving all the other
to the expression $F_{1}^{\pm}(x,t,\varphi)$.

This completes the proof.
\end{proof}

\section{The linear problem corresponding to the problem \eqref{5.19}-\eqref{5.23}.}\label{s6}

In this section we consider the linear problem obtained from the problem \eqref{5.19}-\eqref{5.23}
for a given right-hand sides from corresponding classes. In this case, $|v_{0}^{\pm}(x)|$ is replaced
by $d^{\pm}(x)B^{\pm}(x,t)\sim \lambda(x)\partial_{\lambda}v_{0}^{\pm}(x)$. And, as in the previous section,
we assume that $\lambda(x)$ is extended to all $\overline{\Omega}$ to a smooth function
of the class $H^{3+\gamma}$,

\begin{equation}
\nu\leq\lambda(x),d^{\pm}(x)\leq\nu^{-1},\quad
x\in\overline{\Omega}\setminus \mathcal{N}. \label{6.9}
\end{equation}

Thus, we consider in the domains $\overline{\Omega_{T}^{\pm}}$ the following problem of finding the functions $v^{\pm}(x,t)$,
defined in the domains $\overline{\Omega_{T}^{\pm}}$, and the function $\delta(\omega,t)$, defined on $\Gamma_{T} $, on
the conditions

\[
\frac{\partial v^{\pm}}{\partial t}-\lambda(x)^{\alpha}B^{\pm}(x,t)\nabla^{2}%
v^{\pm}-A^{\pm}(x,t)\left(  \frac{\partial\delta}{\partial
t}-\lambda(x)^{\alpha }B^{\pm}(x,t)\nabla^{2}\delta\right)=
\]

\begin{equation}
=f_{1}^{\pm}(x,t), \ \ (x,t)\in\Omega_{T}^{\pm},
\label{6.1}%
\end{equation}%

\begin{equation}
v^{+}(x,t)=v^{-}(x,t)=0,\quad(x,t)\in\Gamma_{T}, \label{6.2}%
\end{equation}%

\[
k\delta_{t}-\varepsilon\triangle_{\Gamma}\delta+\left(  a^{+}\frac{\partial v^{+}}{\partial\lambda}-a^{-}%
\frac{\partial v^{-}}{\partial\lambda}\right)  -\delta_{\lambda}(a^{+}%
A^{+}(x,t)-a^{-}A^{-}(x,t))=
\]

\begin{equation}
=f_{2}(x,t),\quad(x,t)\in\Gamma_{T}, \label{6.3}%
\end{equation}%

\begin{equation}
v^{\pm}(x,t)=0,\quad(x,t)\in\Gamma_{T}^{\pm}, \label{6.4}%
\end{equation}%

\begin{equation}
v^{\pm}(x,0)=0,\quad\delta(x,0)=0,\quad x\in\overline{\Omega^{\pm}}, \label{6.5}%
\end{equation}%

\begin{equation}
\delta(x,t)=E\delta(\omega,t), \label{6.6}%
\end{equation}
where the extension operator $E$ was defined in the section \ref{s2},
$\triangle_{\Gamma}$ is the Laplace-Beltrami operator on the surface
$\Gamma$ ( compare \cite{1}). We assume that

\begin{equation}
f_{1}^{\pm}(x,t)\in
C_{0,s}^{\gamma,\gamma/2}(\overline{\Omega_{T}^{\pm}}),\quad
f_{2}(x,t)\in C_{0}^{1+\beta-\alpha,\frac{1+\beta-\alpha}{2-\alpha}}(\Gamma_{T}), \label{6.7}%
\end{equation}
$\varepsilon, a^{\pm},k$ are given positive constants,
\begin{equation}
\nu\leq k, a^{\pm},B^{\pm}(x,t),A^{\pm}(x,t) \leq\nu^{-1}, \label{6.8}%
\end{equation}

\begin{equation}
A^{\pm}(x,t), B^{\pm}(x,t)\in C_{s}^{\gamma,\gamma/2}(\overline{\Omega}^{\pm}_{T}) . \label{6.800}%
\end{equation}

For the problem \eqref{6.1}-\eqref{6.6} by the standard method of the freezing of coefficients and
multiplication by smooth cutting functions we can obtain the Schauder a priori estimates of the solution
completely similar to \cite{12} (or \cite{5} in the case of the Stefan problem). At that the model problem,
obtained by the freezing of the coefficients in points of the boundary $\Gamma$ at $t=0$,with the subsequent
local rectification of the boundary, was studied in the section \ref{s3}. At considering such
a model problem the functions $B^{\pm}(x,t)$ and $A^{\pm}(x t)$ are replaced by
the constants $B^{\pm} \equiv  B^{\pm}(x_{0},0)$ and $A^{\pm}\equiv A^{\pm}(x_{0},0)$, $x_{0}\in \Gamma$.
After this, the change of the unknown function

\begin{equation}
u^{\pm}(x,t)=v^{\pm}(x,t)-A^{\pm}\delta\label{CC.1}%
\end{equation}
reduces the problem \eqref{6.1}-\eqref{6.6} with the frozen coefficients and with the flat boundary
exactly to the problem \eqref{M.2}-\eqref{M.6}.

From these model problems associated with the boundary $\Gamma$ we get the estimate
of the function $\delta(x,t)|_{\Gamma}$ and border estimates of the functions $v^{\pm}(x,t)$.
After that, the rest of the model problems associated with a strictly interior points of $\Omega^{\pm}$
are standard because of the condition \eqref{6.6}, and due to the absence of degeneracy of the equations
at these points - see \cite{12}.

Thus, the following is true.

\begin{lemma} \label{LC1}

Suppose that the conditions \eqref{6.7}-\eqref{6.800}. Then for the solution
of the problem \eqref{6.1}-\eqref{6.6} from the class
 $v^{\pm}\in C_{s}^{2+\gamma,\frac{2+\gamma}{2}}(\overline{\Omega}%
_{T}^{\pm})$, $\delta\in
C^{3+\beta-\alpha,1+\frac{1+\beta-\alpha}{2-\alpha} }(\Gamma_{T})$
the following a priori estimate is valid

\[
\left|  v^{\pm}\right|
_{s,\overline{\Omega}_{T}^{\pm}}^{(2+\gamma)}+\left| \delta\right|
_{\Gamma_{T}}^{(2+\beta-\alpha,\frac{2+\beta-\alpha}{2-\alpha
})}+\left|  \delta_{t}\right| _{\Gamma_{T}}^{(1+\beta-\alpha,\frac
{1+\beta-\alpha}{2-\alpha})}+\varepsilon\sum\limits_{i,j=1}^{N-1}\left|
\delta_{\omega_{i}\omega_{j}}\right|
_{\Gamma_{T}}^{(1+\beta-\alpha
,\frac{1+\beta-\alpha}{2-\alpha})}\leq
\]%

\begin{equation}
\leq C_{T}\left(  \left|  f_{1}^{+}\right|  _{s,\overline{\Omega}_{T}^{+}%
}^{(\gamma)}+\left|  f_{1}^{-}\right|  _{s,\overline{\Omega}_{T}^{-}}%
^{(\gamma)}+\left|  f_{2}\right|
_{\Gamma_{T}}^{(1+\beta-\alpha,\frac
{1+\beta-\alpha}{2-\alpha})}\right)  \equiv C_{T}\mathcal{M}(T),\label{CC.2}%
\end{equation}
where the constant $C_{T}$ in \eqref{CC.2} does not depend on $\varepsilon\in
(0,1)$.
\end{lemma}

We now show the solvability of the problem \eqref{6.1}-\eqref{6.6}.

\begin{theorem} \label{TC1}

Suppose that the conditions \eqref{6.7}-\eqref{6.800} are satisfied.
Then for $\varepsilon \in (0,1)$ the problem \eqref{6.1}- \eqref {6.6}
is solvable in the space $v^{\pm}\in
C_{s}^{2+\gamma,\frac{2+\gamma}{2}}(\overline{\Omega
}_{T}^{\pm})$, $\delta\in C^{3+\beta-\alpha,1+\frac{1+\beta-\alpha}{2-\alpha}%
}(\Gamma_{T})$, and the estimate of the solution \eqref{CC.2} is valid.

When $\varepsilon =0 $ the problem \eqref{6.1}-\eqref{6.6} is solvable
in the space $v^{\pm}\in C_{s}^{2+\gamma,\frac{2+\gamma}{2}}(\overline{\Omega}%
_{T}^{\pm})$, $\delta\in C^{2+\beta-\alpha,\frac{2+\beta-\alpha}{2-\alpha}%
}(\Gamma_{T})$ ,$\delta_{t}\in
C^{1+\beta-\alpha,\frac{1+\beta-\alpha }{2-\alpha}}(\Gamma_{T})$,
and the estimate \eqref{CC.2} without the term with $\varepsilon$
is valid.

\end{theorem}

\begin{proof}

Define the linear operator  $M:\delta\rightarrow
v^{\pm}\rightarrow M\delta$ which maps a function $\delta\in
C^{2+\beta-\alpha,\frac{2+\beta-\alpha}{2-\alpha} }(\Gamma_{T})$
fist to the functions $v^{\pm}$, as the solution of the problem
\eqref{6.1}, \eqref{6.2}, \eqref{6.4}, \eqref{6.5} with the given
function $\delta$ in \eqref{6.1}, and then the functions $v^{\pm}$
the operator $M$ maps to the function $M\delta$, which is determined
from the condition \eqref{6.3} with the given $v^{\pm}$ and $\delta_{\lambda}$,
that is the function $M\delta$ is the solution of the problem

\begin{equation}
k(M\delta)_{t}-\varepsilon\Delta_{\Gamma}(M\delta)=f_{2}-\left(  a^{+}%
\frac{\partial v^{+}}{\partial\lambda}-a^{-}\frac{\partial v^{-}}%
{\partial\lambda}\right)  +\delta_{\lambda}(a^{+}A^{+}-a^{-}A^{-}),\label{CC.3}%
\end{equation}%

\[
M\delta(\omega,0)=0.
\]

By the theorem \ref{TB1} this operator is well defined,
and with $\varepsilon>0$, by the known properties
of the problem \eqref{CC.3},

\begin{equation}
\left|  M\delta\right|
_{\Gamma_{T}}^{(3+\beta-\alpha,1+\frac{1+\beta-\alpha
}{2-\alpha})}\leq C_{\varepsilon,T}\left(  \left|  \delta\right|
_{\Gamma
_{T}}^{(2+\beta-\alpha,\frac{2+\beta-\alpha}{2-\alpha})}+\mathcal{M}(T)\right)
,\label{CC.4}%
\end{equation}%

\begin{equation}
\left|  M\delta_{2}-M\delta_{1}\right|
_{\Gamma_{T}}^{(3+\beta-\alpha
,1+\frac{1+\beta-\alpha}{2-\alpha})}\leq C_{\varepsilon,T}\left|
\delta _{2}-\delta_{1}\right|
_{\Gamma_{T}}^{(2+\beta-\alpha,\frac{2+\beta-\alpha
}{2-\alpha})}.\label{CC.5}%
\end{equation}

Consequently, by \eqref{B.6},

\[
\left|  M\delta_{2}-M\delta_{1}\right|
_{\Gamma_{T}}^{(2+\beta-\alpha
,\frac{2+\beta-\alpha}{2-\alpha})}\leq CT^{\mu}\left|
M\delta_{2}-M\delta _{1}\right|
_{\Gamma_{T}}^{(3+\beta-\alpha,1+\frac{1+\beta-\alpha}{2-\alpha
})}\leq
\]%

\begin{equation}
\leq C_{\varepsilon,T}T^{\mu}\left|  \delta_{2}-\delta_{1}\right|
_{\Gamma_{T}}^{(2+\beta-\alpha,\frac{2+\beta-\alpha}{2-\alpha})}.\label{CC.6}%
\end{equation}

Thus, for a sufficiently small $T=T_{\varepsilon}$ the operator $M$ is a contraction
on $C_{0}^{2+\beta-\alpha,\frac{2+\beta-\alpha}{2-\alpha}}(\Gamma_{T})$ and therefore
has a unique fixed point, which by \eqref{CC.4}, belongs also to the space
$C_{0}^{3+\beta-\alpha, 1+\frac{1+\beta-\alpha}{2-\alpha}}(\Gamma_{T})$ and together with
the corresponding $v^{\pm}$ gives the solution of the problem. The estimate of the solution
is given by the lemma \ref{LC1}. Moving now step by step up the axis $Ot$ as in \cite{12},
we obtain the theorem with $\varepsilon>0$ for any $T>0$.

Further, by the estimate \eqref{CC.2}, considering the sequence of the
solutions $v^{\pm}_{\varepsilon}$, $\delta_{\varepsilon}$, $\varepsilon\rightarrow 0$,
we see that this sequence is compact in the spaces
 $C_{0,s}^{2+\overline{\gamma},\frac{2+\overline{\gamma}}{2}}%
(\overline{\Omega}_{T}^{\pm})$ and $C_{0}^{2+\overline{\beta}-\alpha
,\frac{2+\overline{\beta}-\alpha}{2-\alpha}}(\Gamma_{T})$
correspondingly for any $\overline {\gamma}<\gamma$,
$\overline{\beta}=\overline{\gamma}(1-\alpha/2)$. The passing to the limit
of this sequence in the spaces
 $C_{0,s}^{2+\overline{\gamma},\frac{2+\overline{\gamma}}{2}%
}(\overline{\Omega}_{T}^{\pm})$ and
$C_{0}^{2+\overline{\beta}-\alpha
,\frac{2+\overline{\beta}-\alpha}{2-\alpha}}(\Gamma_{T})$,
$\delta_{t}\in
C_{0}^{2+\overline{\beta}-\alpha,\frac{2+\overline{\beta}-\alpha}{2-\alpha}%
}(\Gamma_{T})$ gives the solution of the problem \eqref{6.1}-
\eqref{6.6} for $\varepsilon=0$. Besides, as it follows from the uniform
in $\varepsilon$ estimate \eqref{CC.2}, the limit function $v^{\pm}$ and $\delta$
belong to the spaces
$C_{0,s}^{2+\gamma,\frac{2+\gamma}{2}}(\overline{\Omega}_{T}
^{\pm})$ and
$C_{0}^{2+\beta-\alpha,\frac{2+\beta-\alpha}{2-\alpha}}(\Gamma
_{T})$, $\delta_{t}\in
C_{0}^{2+\beta-\alpha,\frac{2+\beta-\alpha}{2-\alpha}
}(\Gamma_{T})$ correspondingly.

Thus, the theorem \ref{TC1} is proved.

\end{proof}

\section{Completion of the proof of the theorem \ref{T1.1}. }\label{s7}

Доказательство теоремы \ref{T1.1} завершается по схеме, описанной
в параграфе \ref{s5}.

Определим на пространстве $\mathcal{H}_{0}$ ($\mathcal{H}_{0}$
определено в \eqref{5.13})  нелинейный оператор
$\mathcal{F}(\varphi)$, $\varphi=(\theta^{+},\theta^{-},\delta)\in
\mathcal{H}_{0}$  в \eqref{5.19}- \eqref{a5.23}, который каждому
заданному $\varphi$ в нелинейных правых частях соотношений
\eqref{5.19}, \eqref{5.21} ставит в соответствие решение линейной
задачи, определяемой левыми частями этих соотношений. При этом из
теоремы \ref{TC1} и леммы \ref{L5.1} следует, что оператор
$\mathcal{F}(\varphi)$ обладает следующими свойствами на шаре
$\mathcal{B}_{r}=\{\varphi: \ \|\varphi\|\leq r\}\subset
\mathcal{H}_{0}$ достаточно малого радиуса:

\begin{equation}
\|\mathcal{F}(\varphi)\|_{\mathcal{H}}\leq
C(T^{\alpha/2}+r)\|\varphi\|_{\mathcal{H}}, \label{75.1}
\end{equation}

\begin{equation}
\|\mathcal{F}(\varphi_{1})-\mathcal{F}(\varphi_{2})\|_{\mathcal{H}}\leq
C(T^{\alpha/2}+r)\|\varphi_{1}-\varphi_{2}\|_{\mathcal{H}}.
\label{75.2}
\end{equation}

Нетрудно видеть, что из соотношений \eqref{75.1} и \eqref{75.2}
следует, что при достаточно малых $T$ и $r$ оператор
$\mathcal{F}(\varphi)$ отображает замкнутый шар $\mathcal{B}_{r}$
в себя и является там сжимающим. Единственная неподвижная точка
этого оператора и дает решение исходной нелинейной задачи со
свободной границей. Тем самым теорема \ref{T1.1} доказана.
$\square$





\end{document}